\documentclass[twoside]{article}
\usepackage[dvipsnames]{xcolor}
\usepackage{amsmath, amssymb, amsthm, datetime, tikz, scalerel, esvect, mathtools}
\usepackage[explicit]{titlesec}
\usepackage{enumitem}
    \setlist{nosep}
    \setlist[itemize]{label=$\circ$}

\usepackage[hyperfootnotes=false,pdftitle={Dynamics and entropy of S-graph shift},colorlinks=true,urlcolor=black,linkcolor=black,citecolor=SeaGreen]{hyperref}
\usepackage[capitalize,noabbrev]{cleveref}
\usepackage[numbers,sort]{natbib}
\makeatletter
\newtheorem*{rep@theorem}{\rep@title}
\newcommand{\newreptheorem}[2]{%
\newenvironment{rep#1}[1]{%
 \def\rep@title{#2 \ref{##1}}%
 \begin{rep@theorem}}%
 {\end{rep@theorem}}}
\makeatother

\newtheorem{theorem}{Theorem}[section]
\newreptheorem{theorem}{Theorem}
\newtheorem{lemma}[theorem]{Lemma}
\newtheorem{corollary}[theorem]{Corollary}
\newtheorem{proposition}[theorem]{Proposition}

\theoremstyle{definition}
\newtheorem{definition}[theorem]{Definition}
\newtheorem{exampleprimitive}[theorem]{Example}
\newtheorem{question}[theorem]{Question}

\theoremstyle{remark}
\newtheorem*{remark}{Remark}

\numberwithin{equation}{section}
\allowdisplaybreaks

\newenvironment{example}{
    
    \pushQED{\qed}
    \begin{exampleprimitive}
}{
    \popQED
    \end{exampleprimitive}
}

\newcommand{\N}{\mathbb{N}}
\newcommand{\Z}{\mathbb{Z}}
\newcommand*{\smallrightarrow}{{\scalebox{0.5}[0.75]{$\rightarrow$}}}
\newcommand*{\smallleftarrow}{{\scalebox{0.5}[0.75]{$\leftarrow$}}}

\renewcommand{\epsilon}{\varepsilon}
\renewcommand{\S}{\mathcal{S}}

\makeatletter
    \renewcommand\@makefntext[1]{\leftskip=0em\hskip-0em\@makefnmark\,#1}
    \def\blankfootnote{\xdef\@thefnmark{\ \,}\@footnotetext}
\makeatother
\DeclareMathOperator{\tr}{tr}
\DeclareMathOperator{\fc}{fc}

\DeclareMathOperator{\diag}{diag}
\DeclareMathOperator{\sgn}{sgn}
\DeclareMathOperator{\spec}{spec}

\definecolor{light blue}{RGB}{115,180,255}

\usetikzlibrary{arrows,shapes,calc,animations}
\tikzset{every node/.style = {shape=circle,draw,minimum size=1.3em, inner sep=0.5mm}}
\tikzset{shorten > = 1.5pt, shorten < = 1.5pt}
\tikzset{edge/.style={->, > = latex'}}

\usepackage{geometry}
\newcommand{\inserttitle}{dynamics and entropy of $\scaleobj{0.8}{\mathcal{S}}$-graph shifts} 
\newdateformat{dmy}{%
    \THEDAY\, \monthname[\THEMONTH] \THEYEAR}

\title{\vspace{-1.5em} \inserttitle}
\author{Travis Dillon}

\usepackage{fancyhdr}
\fancypagestyle{firstpage}
{
   \fancyhf{}
}

\begin{document}

\thispagestyle{firstpage}

\makeatletter

\null
\vspace{-2\baselineskip}

\noindent
    \vbox{%
        \hsize\textwidth
        \linewidth\hsize
        {
          \hrule height 1.3\p@
          \vskip 0.2in
          \vskip -\parskip%
        }
        \centering
        {\LARGE\sc \inserttitle\par}
        {
          \vskip 0.21in
          \vskip -\parskip
          \hrule height 1.3\p@
          \vskip 0.10in%
        }
    }
\makeatother
\begin{center}\large\renewcommand{\thefootnote}{\fnsymbol{footnote}}
    Travis Dillon\footnote{Massachusetts Institute of Technology, Cambridge, Massachusetts, USA}\\[1.5em]
    
\end{center}

\blankfootnote{AMS Subject Classification: 37B10}\blankfootnote{Key words: symbolic dynamics, shift space, $S$-gap shift, entropy, zeta function, directed graph}

\setcounter{footnote}{0}\renewcommand{\thefootnote}{\arabic{footnote}}

\begin{abstract}
\noindent
$S$-gap shifts are a well-studied class of shift spaces, which has led to several proposed generalizations. This paper introduces a new class of shift spaces called $\mathcal{S}$-graph shifts whose essential structure is encoded in a novel way, as a finite directed graph with a set of natural numbers assigned to each vertex. $\mathcal{S}$-graph shifts contain $S$-gap shifts and their generalizations, as well as all vertex shifts and SFTs, as special cases, thereby providing a method to study these shift spaces in a uniform way. The main result in this paper is a formula for the entropy of any $\mathcal{S}$-graph shift, which, by specialization, resolves a problem proposed by Matson and Sattler. A second result establishes an explicit formula for the zeta functions of $\mathcal{S}$-graph shifts. Additionally, we show that every entropy value is obtained by uncountably many $\mathcal{S}$-graph shifts. 
\end{abstract}

\section{Introduction}

Let $\mathcal A$ be a finite set. A \textit{shift space} is a closed set $X \subseteq \mathcal A^{\Z}$ that is invariant under the \textit{shift map} $\sigma\colon \mathcal A^{\Z} \to \mathcal A^{\Z}$ defined in section \cref{sec:background}. $S$-gap shifts comprise a well-studied class of shift spaces; among other reasons, they are especially useful in producing simple concrete examples of shift spaces with a variety of dynamical properties. They also have an especially rich structure, and analyses of $S$-gap shifts have drawn from many areas, including algebra \cite{computations-S-gap}, topology \cite{dynamics-and-topology}, number theory \cite{dynamical-properties, equivalencies}, and ergodic theory \cite{intrinsic-ergodicity}.

Each $S$-gap shift is uniquely defined by a subset $S$ of the natural numbers. Motivated by coding theory, Dastjerdi and Shaldehi \cite{S-prime} introduced $(S,S')$-gap shifts, a generalization of $S$-gap shifts that are defined by a pair of subsets of the natural numbers, and studied their dynamical properties. Later, Matson and Sattler \cite{S-limited} further generalized $(S,S')$-gap shifts, in two different ways, to ordered and unordered $\mathcal{S}$-limited shifts; any $\mathcal{S}$-limited shift is specified by a finite tuple of subsets of the natural numbers. They extended many of the dynamical results for $S$- and $(S,S')$-gap shifts in \cite{dynamics-and-topology} and \cite{S-prime} to $\mathcal{S}$-limited shifts.

Occupying a central role in each of these investigations was the determination of an entropy formula for the class of shift spaces under consideration. Entropy is one of the key invariants in symbolic dynamics; roughly speaking, entropy provides a measure of the complexity of a shift space. Spandl \cite{spandl} first computed the entropy of $S$-gap shifts in 2007, obtaining a formula in terms of the root of a certain power series (see \cref{thm:S-gap-entropy}).

Unfortunately, Spandl's proof was incomplete. Recently, Garc\'ia-Ramos and Pavlov filled the gap in Spandl's proof and provided a separate one of their own \cite[Corollaries 3.20 and 3.21]{extender-sets}. For a period of several years before this paper, the canonical proofs of the $S$-gap entropy formula were the two alternative ones presented on Vaughn Climenhaga's blog.\footnote{\url{https://vaughnclimenhaga.wordpress.com/2014/09/08/entropy-of-s-gap-shifts/}} Matson and Sattler extended one of Climenhaga's proofs to produce an entropy formula for ordered $\mathcal{S}$-limited shifts. They note, however, that their approach is insufficient to prove a formula for unordered $\mathcal{S}$-limited shifts.

Finding and proving such a formula is the motivation for this work. This paper introduces a much larger class of shift spaces, called $\mathcal{S}$-graph shifts, that contains ordered and unordered $\mathcal{S}$-limited shifts, as well as all subshifts of finite type, as special cases. Each $\mathcal{S}$-graph shift is defined by a finite set-equipped graph, formed by assigning a subset of the natural numbers to each vertex. As described in \cref{sec:entropy-formula}, we can assign a matrix-valued function $B(x)$, called the generating matrix, to each $\mathcal{S}$-graph shift. Our first main result is an entropy formula for $\mathcal{S}$-graph shifts.

\begin{theorem}\label{thm:entropy_eigen=1}
    Let $X$ be an $\mathcal{S}$-graph shift, $B(x)$ its generating matrix, and $\rho(x)$ the spectral radius of $B(x)$. The equation $\rho(x) = 1$ has a unique positive solution $x = \lambda$, and the entropy of $X$ is $\log \lambda^{-1}$.
\end{theorem}

In particular, \cref{thm:entropy_eigen=1} provides an entropy formula for unordered $\mathcal{S}$-limited shifts by specialization, thereby resolving Matson and Sattler's problem in a much more general setting. \cref{thm:entropy_eigen=1} may also be formulated in terms of a power series which simplifies to the one in \cref{thm:S-gap-entropy} when $X$ is an $S$-gap shift (see \cref{thm:entropy_sumovercycles}).

Recently, Ronnie Pavlov obtained an entropy formula for coded shift spaces \cite{coded-entropy}. Every irreducible $\S$-graph shift is coded, so Pavlov's formula may in principle be applied to them. However, this paper introduces not just a new entropy formula but also a new set of linear algebraic tools with which to analyze these shift spaces. For example, we use properties of the matrix $B(x)$ to prove a strong result on the distribution of entropy among $\mathcal{S}$-graph shifts.

\begin{theorem}\label{thm:every-entropy}
    Let $\lambda > 1$. For each $q \geq 2\lceil \lambda \rceil - 1$, there are uncountably many pairwise non-conjugate $\mathcal{S}$-graph shifts on $q$ vertices that have the specification property and entropy $\log \lambda$.
\end{theorem}

Baker and Ghenciu, in Proposition 4.2 and Theorem 4.3 of \cite{dynamical-properties}, provide similar results that are restricted to $S$-gap shifts.

In the final section of this paper, we turn our attention to another central invariant of shift spaces: the zeta function, which encodes all the periodic point information of a shift space. Our final main result establishes an explicit formula for the zeta function of any $\S$-graph shift, again in terms of the generating matrix $B(x)$.

\begin{theorem}\label{thm:big-zeta}
    Let $X$ be an $\S$-graph shift and $B(x)$ be its generating matrix. The zeta function of $X$ is
    \[
        \zeta_X(t) = \frac{1}{(1-t)^{p_1(X)}\det\!\big(I-B(t)\big)}.
    \]
\end{theorem}

A straightforward corollary of this theorem is a simple explicit formula for the zeta function of any ordered $\S$-limited shift (see \cref{thm:ordered-s-limited-zeta}). Even this restricted result greatly extends the work in \cite{computations-S-gap} and \cite{S-prime}, which established the zeta functions of $S$- and $(S,S')$-gap shifts, respectively.

Together, Theorems \ref{thm:entropy_eigen=1}, \ref{thm:every-entropy}, and \ref{thm:big-zeta} comprise a compelling case for the power of the generating matrix as a tool in the analysis of shift spaces. (Though \cref{thm:every-entropy} doesn't mention the generating matrix, the properties of this matrix are central in the proof of the theorem.) It seems likely that the generating matrix will be able to shed light on other key questions, such as the determination of conjugate shifts. Though the conjugacy problem was solved completely for $S$-gap shifts in \cite{dynamics-and-topology}, understanding is limited even for $(S,S')$-gap shifts. Dastjerdi and Shaldehi find one necessary and several sufficient conditions for conjugacy in \cite{S-prime}, which Matson and Sattler extend to ordered $\S$-limited shifts in \cite{S-limited}; this seems to be the extent of the current progress on this problem.

The paper is organized as follows. In \cref{sec:background}, we describe the relevant background from symbolic dynamics and basic definitions for $\mathcal{S}$-graph shifts. \cref{sec:dynamical-properties} provides characterizations of the $\mathcal{S}$-graph shifts that have key dynamical properties, extending the work in \cite{dynamical-properties} and \cite{S-limited}. The proof of \cref{thm:entropy_eigen=1} comprises \cref{sec:entropy-formula}. \cref{sec:graph-operations} leverages \cref{thm:entropy_eigen=1} to characterize $\S$-graph shifts with entropy 0, while  \cref{sec:entropy-construction} proves \cref{thm:every-entropy}. Together, Sections \ref{sec:entropy-formula}, \ref{sec:graph-operations}, and \ref{sec:entropy-construction} introduce a collection of new techniques for analyzing a large class of shift spaces. In \cref{sec:intrinsic-ergodicity}, we prove that every $\S$-graph shift is intrinsically ergodic. \cref{sec:zeta} proves \cref{thm:big-zeta}. For ease of reading the body of this paper, proofs of the more technical results are collected in \cref{sec:appendix}.

\section{Background and definitions}\label{sec:background}

Proofs of the assertions made in this section can be found in the introductory symbolic dynamics text by by Lind and Marcus \cite{Lind-Marcus}. Examples of the definitions appear following the definition of $S$-gap shifts (\cref{def:S-gap}).

Throughout this paper, we let $\N_0$ denote the set of nonnegative integers and $\N$ denote the set of positive integers. The symbol $\mathcal A$ will always denote a finite set called the \textit{alphabet}, whose elements are called \textit{letters}. The \textit{full shift on $\mathcal A$}, denoted $\mathcal A^\Z$, is the set of all bi-infinite strings of elements in $\mathcal A$. Equivalently, the full shift is the set of all sequences $(x_n)_{n=-\infty}^\infty$ with coordinates in $\mathcal A$; this index notation will be used throughout the paper.

The \textit{shift map} $\sigma\colon \mathcal A^\Z \to \mathcal A^\Z$ sends the point $x \in \mathcal A^\Z$ to the point $y \in \mathcal A^\Z$ that satisfies $y_i = x_{i+1}$ for each $i \in \Z$. We equip the full shift with a metric $d$, setting $d(x,y) = 2^{-(k+1)}$ when $k$ is the greatest integer such that $x_{-k}\cdots x_{k} = y_{-k} \cdots y_{k}$ and $d(x,y) = 0$ if no such $k$ exists. (If $x_0 \neq y_0$, then $d(x,y) = 1$.) With respect to this metric, $\mathcal A^\Z$ is a compact metric space and $\sigma$ is continuous. A \textit{subshift} or \textit{shift space} over $\mathcal A$ is a subset $X$ of $\mathcal A^\Z$ that is closed and shift-invariant (that is, $\sigma(X) = X$).

A \textit{word} over $\mathcal A$ is a finite sequence of letters. We denote the concatenation of two words by juxtaposition, so if $x = x_1\cdots x_n$ and $y = x_{n+1}\cdots x_{n+k}$, then $xy = x_1\cdots x_{n+k}$. The $n$-fold concatenation of $u$ is denoted by $u^n$. We say that a word $\alpha = \alpha_1\cdots\alpha_k$ \textit{appears} in a point $x \in \mathcal A^\Z$ if $\alpha$ is a consecutive subsequence of $x$, that is, $\alpha = x_n\cdots x_{n+k-1}$ for some $n \in \Z$.

\begin{definition}
    For each set $\mathcal F$ of words over $\mathcal A$, we define
    \[ X_\mathcal F = \{x \in \mathcal A^\Z : \text{no element of $\mathcal F$ appears in } x. \} \]
\end{definition}

It turns out that the subspaces defined by forbidden words are exactly the subshifts of $\mathcal A^\Z$: Each set $X_\mathcal F$ is both closed and shift-invariant, and for each subshift $X$ there is a set of forbidden words $\mathcal F$ (not necessarily unique) such that $X = X_\mathcal F$. We say that a subshift $X$ is a \textit{subshift of finite type} (or SFT) if there is a finite set $\mathcal F$ such that $X = X_\mathcal F$.

The set of \textit{allowable words of length $n$}, denoted $\mathcal B_n(X)$, is the set of words of $n$ letters that appear in some word in $X$. The \textit{language} of $X$ is the set $\mathcal L(X) = \bigcup_{n=1}^\infty \mathcal B_n(X)$ of all allowable words.

\begin{theorem}\label{thm:language-equality}
    Two shift spaces are equal if and only if their languages are.
\end{theorem}

The entropy of a shift space is a measure of the size of its language, which, to a certain extent, provides a measure of the complexity of the shift space.

\begin{definition}
    The \textit{\emph{(}topological\emph{)} entropy} of a shift space $X$ is
    \[ h(X) = \lim_{n\to\infty} \frac{1}{n} \log \lvert \mathcal B_n(X)\rvert. \]
\end{definition}

This limit always exists if $X$ is nonempty, in which case $1 \leq \lvert \mathcal B_n(X)\rvert \leq \lvert \mathcal A\rvert^n$; consequently $0 \leq h(X) \leq \log\, \lvert \mathcal A \rvert$.

A point $x \in \mathcal A^\Z$ is called \textit{periodic} with period $n$ if $\sigma^n(x) = x$; the point $x$ has \textit{least period} $n$ if furthermore $\sigma^k(x) \neq x$ for each $k \in \{1,\dots,n-1\}$. Any word $x=x_1x_2\cdots x_n$ gives rise to a periodic point $x^\infty$ defined by $x^\infty_j = x_{j \text{ mod } n}$. The number of points in $X$ with period $n$ is denoted $p_n(X)$, and the number of points in $X$ with least period $n$ is denoted $q_n(X)$. The \textit{zeta function} is a power series that gathers the information on the periodic points of a shift space into a single object:
\[
    \zeta_X(t) = \exp\!\left(\,\sum_{n=1}^\infty \frac{p_n(X)}{n}t^n\right).
\]

Given two shift spaces $X$ and $Y$, a \textit{homomorphism} from $X$ to $Y$ is a continuous map $\phi\colon X \to Y$ that commutes with the shift map. A \textit{conjugacy} is a bijective homomorphism with continuous inverse; if there is a conjugacy from $X$ to $Y$, we call $X$ and $Y$ \textit{conjugate} shift spaces and write $X \cong Y$. If $X$ and $Y$ are conjugate, then $h(X) = h(Y)$. Moreover, $p_n(X) = p_n(Y)$ and $q_n(X) = q_n(Y)$ for every $n \in \N$.

\begin{definition}
    Let $X$ be a shift space, $m, n \in \N_0$, and $\Phi\colon \mathcal B_{m+n+1}(X) \to \mathcal A$. The \textit{sliding block code} with memory $m$ and anticipation $n$, also called an $(m+n+1)$-block code, is the map $\phi\colon X \to Y$ defined by
    \[ \phi(x)_i = \Phi(x_{i-m}\cdots x_{i+n}). \]
\end{definition}

Every sliding block code is a homomorphism. The Curtis-Hedlund-Lyndon Theorem is the surprising result that the converse is true: Every homomorphism between two shift spaces is a sliding block code.

A shift space is called \textit{sofic} if it is the image of an SFT under a homomorphism. The \textit{follower set} of a word $\alpha \in \mathcal L(X)$ is $F_X(\alpha) = \{\beta \in \mathcal L(X) : \alpha\beta \in \mathcal L(X) \}$.

\begin{theorem}\label{thm:sofic-follower-sets}
    A shift space $X$ is sofic if and only if it has a finite number of distinct follower sets.
\end{theorem}

For example, the full shift has only one follower set: For any $\alpha \in \mathcal A^\Z$, we have $F_{\mathcal A^\Z}(\alpha) = \mathcal A^\Z$. So $\mathcal A^\Z$ is sofic.

\begin{definition}\label{def:S-gap}
    For each $S \subseteq \N_0$, the $S$-gap shift $X(S)$ is the subshift of $\{0,1\}^\Z$ obtained by taking the closure of the set
    \[ \{ \cdots 10^{s_{-1}}10^{s_0}10^{s_1}1 \cdots
        \, :\mspace{1.5mu} s_i \in S \text{ for each } i \in \N \} \]
    with respect to the metric $d$.
\end{definition}

Taking the closure has no effect when $S$ is finite; when $S$ is infinite, it allows strings to begin or end with infinitely many 0's.

\begin{example}[Golden mean shift]\label{ex:golden-mean}
    The shift space $X(\N)$ is called the \textit{golden mean shift}; it is the subshift of $\{0,1\}^{\Z}$ consisting of all words without consecutive 1's. The name is derived from the entropy of $X(\N)$: Using \cref{thm:S-gap-entropy}, it is straightforward to show that $h(X(\N)) = (1+\sqrt{5})/2$. If we set $\mathcal F = \{11\}$, then $X(\N) = X_{\mathcal F}$, so the golden mean shift is an SFT; since any SFT is sofic, $X(\N)$ is sofic. Finally, $\mathcal B_3(X(\N)) = \{000,001,010,100,101\}$.
\end{example}

\begin{example}[Even shift]
    The shift space $X(2\N_0)$ is called the \textit{even shift}. Since no finite set of words can forbid $10^{2n+1}1$ for every $n \in \N_0$ while simultaneously allowing $10^{2n}1$ for every $n \in \N_0$, the even shift is not an SFT. However, every follower set of $X(2\N_0)$ is equal to one of $F_{X(2\N_0)}(0)$, $F_{X(2\N_0)}(1)$, or $F_{X(2\N_0)}(10)$, so the even shift is sofic. Alternatively, define $\Phi\colon \mathcal \{0,1\}^3 \to \{0,1\}$ by
    \[  \Phi(x_1x_2x_3) = \begin{cases}
            0 &\text{if } x_1x_2 = 01 \text{ or } x_2x_3 = 01\\
            1 &\text{otherwise.}
        \end{cases} \]
    Let $\phi\colon X(\{0,1\}) \to X(2\N_0)$ be the sliding block code with memory 1 and anticipation 1 induced by $\Phi$. Since $X(\{0,1\})$ is an SFT (take $\mathcal F = \{00\}$) and $\phi$ is surjective, the even shift is sofic.
\end{example}

If we allow restrictions on blocks of consecutive 1's as well 0's, we obtain $(S,S')$-gap shifts  \cite{S-prime}.

\begin{definition}
    Let $S,S' \subseteq \N$. The $(S,S')$-gap shift $X(S,S')$ is the subshift of $\{0,1\}^\Z$ obtained by taking the closure of
    \[ \{ \cdots 1^{s'_{-1}}0^{s_{-1}}1^{s'_0}0^{s_0}1^{s'_1}0^{s_1} \cdots
        \,:\mspace{1.5mu} s_i \in S \text{ and } s'_i \in S' \text{ for each } i \in \N \}. \]
\end{definition}

The specialization to $S$-gap shifts has two cases: If $0 \notin S$, then $X(S) = X(S,\{1\})$; if $0 \in S$, then $X(S) = X(S,\N)$.  $\mathcal{S}$-limited shifts are extensions of $(S,S')$-gap shifts to alphabets with more than two letters \cite{S-limited}. We denote by $[n]$ the set $\{1,2,\dots,n\}$.

\begin{definition}
    Let $n \geq 2$ and $\mathcal S = (S_1,\dots, S_n)$ be an $n$-tuple of nonempty subsets of $\N$. We set $G_\mathcal S = \{1^{m_1}2^{m_2}\cdots n^{m_n} : m_i \in S_i \text{ for each } i \in [n] \}$. The \textit{ordered $\mathcal{S}$-limited shift}, denoted $\vv{X}(\mathcal S)$, is the subshift of $\{1,2,\dots,n\}^\Z$ obtained by taking the closure of
    \[ \{\cdots x_{-1}x_0x_1 \cdots
            \,:\mspace{1.5mu} x_i \in G_{\mathcal S} \text{ for each } i \in \Z\}. \]
    The \textit{unordered $\mathcal{S}$-limited shift}, denoted $X(\mathcal S)$, is the subshift of $\{1,2,\dots,n\}^\Z$ obtained by taking the closure of
    \[ \{\cdots a_{-1}^{m_{-1}}a_0^{m_0}a_1^{m_1}\cdots
            \,:\mspace{1.5mu} a_i \in [n]\text{, } m_i \in S_{a_i} \text{, and } a_{i+1} \neq a_i \text{ for each } i \in \Z \}. \]
\end{definition}

Ordered and unordered $\mathcal{S}$-limited shifts coincide when $n=2$, yielding $(S,S')$-gap shifts.\footnote{In \cite{S-limited}, ordered and unordered $\mathcal{S}$-limited shifts are referred to as ``$\mathcal{S}$-limited shifts'' and ``generalized $\mathcal{S}$-limited shifts'' and denoted by $X(\mathcal{S})$ and $X_\mathcal{S}$, respectively. Since most $\mathcal{S}$-limited shifts are not generalized $\mathcal{S}$-limited shifts, this paper instead uses the modifiers ``ordered'' and ``unordered.'' The notation was also adjusted to avoid overloading shift space construction using subscripts.}

To describe $\mathcal{S}$-graph shifts, we need the language of graph theory. Unless otherwise mentioned, every graph in this paper is finite, directed, and simple (no loops and no multiple edges). The edge and vertex sets of a graph $G$ are denoted by $E(G)$ and $V(G)$, respectively. A vertex of $G$ is \textit{essential} if its in-degree and out-degree are both at least 1; a graph is \textit{essential} if every vertex is. The underlying essential graph of a graph $G$ is obtained by repeatedly removing all nonessential vertices until none remain. A \textit{walk} on $G$ is a finite sequence of vertices $(v_1,\dots, v_n)$ such that $(v_i,v_{i+1}) \in E(G)$ for each $i \in \{1,2,\dots,n-1\}$. We call $G$ \textit{irreducible}\phantomsection\label{def:irreducible-graph} if there is a walk from $u$ to $v$ for every $u,v \in V(G)$.  The set of all walks on $G$ is denoted $W(G)$. A walk $(v_1,\dots,v_n)$ is called \emph{closed} if $n\geq 2$ and $(v_n,v_1) \in E(G)$;\footnote{This differs from the usual definition of a closed walk in a graph, in which the last vertex \textit{equals} the first vertex, but our definition simplifies the correspondence between closed walks and dynamical properties in this paper.} a closed walk is a \emph{cycle} if all vertices in the walk are distinct.

Given a graph $G$, the \textit{vertex shift on $G$} is the set
\[ \{ \cdots u_{-1}u_0u_1 \cdots
        \,:\mspace{1.5mu} u_i \in V(G) \text{ and } (u_i,u_{i+1}) \in E(G) \text{ for each } i \in \Z\}. \]
A more robust construction uses vertex labels.

\begin{definition}
    A \textit{labelled graph} is a pair $(G,\ell)$, where $G$ is a graph, $\mathcal A$ is a finite alphabet, and $\ell\colon V(G) \to \mathcal A$. The shift space generated by a labelled graph $\mathcal G$, denoted $X_\mathcal G$, is the set
    \[ \{ \cdots \ell(u_{-1})\ell(u_0)\ell(u_1) \cdots
        \,:\mspace{1.5mu} u_i \in V(G) \text{ and } (u_i,u_{i+1}) \in E(G) \text{ for each } i \in \Z \}. \]
\end{definition}

Every shift space obtained from a labelled graph is sofic, and for every sofic shift $X$ there is a labelled graph $\mathcal G$ so that $X = X_\mathcal G$.

Having set the scene, we now introduce $\mathcal{S}$-graph shifts, which fuse ideas from $S$-gap shifts and vertex shifts.

\begin{definition}
    A \textit{set-equipped graph} is a pair $(G,\mathcal S)$ where $G$ is a directed graph and $\mathcal S = (S_v)_{v \in V(G)}$ is an indexed collection of nonempty subsets of $\N$. If $\mathcal G = (G,\mathcal S)$ is a set-equipped graph, the $\mathcal{S}$-graph shift $X(\mathcal G) = X(G,\mathcal S)$ is the closure (in the full shift $V(G)^\Z$) of the set
    \[ \{ \cdots u_{-1}^{n_{-1}}u_0^{n_0}u_1^{n_1}
        \,:\mspace{1.5mu} u_i \in V(G)\text{, } n_i \in S_{u_i}\text{, and } (u_i,u_{i+1}) \in E(G) \text{ for each } i \in \Z \}. \]
\end{definition}

Every vertex shift is an $\S$-graph shift (by associating the singleton $\{1\}$ with every vertex); since every SFT is conjugate to a vertex shift \cite[Theorem 2.3.2]{Lind-Marcus}, every SFT is (conjugate to) an $\S$-graph shift, as well. If $G$ is a directed cycle, then $X(G,\mathcal S)$ is the ordered $\mathcal{S}$-limited shift $\vv{X}(\mathcal{S})$; if $G$ is a complete graph (with directed edges from each vertex to every other), then $X(G,\mathcal S)$ is the unordered $\mathcal S$-limited shift $X(\mathcal S)$. Intuitively, these two examples represent the extreme cases of $\S$-graph shifts: any essential connected directed graph has at least as many edges as a directed cycle and at most as many edges as a complete graph.

There are many $\S$-graph shifts that do not lie at these extremes. As the following example shows, $\S$-limited shift comprise a proper subset of $\S$-graph shifts.

\begin{example}
    Let $G$ be the graph
    \begin{center}
    \begin{tikzpicture}
        \foreach \a [count=\x] in {a,b,c}
            \node (\a) at (\x,0) {$\a$};
        \foreach \a/\b in {a/b, b/a, b/c, c/b}
            \draw[bend left, ->] (\a) to (\b);
    \end{tikzpicture}
    \end{center}
    and set $S_a = S_b = S_c = \N$. Denoting by $X$ the resulting $\S$-graph shift, we note that $q_1(X) = 3$ and $q_2(X) = 4$. Any ordered $\S$-limited shift $Y$ with $q_1(Y) = 3$ has at least 3 letters, while any ordered $\S$-limited shift with $q_2(Y) > 0$ has at most 2 letters. Moreover, if $Z$ is an unordered $\S$-limited shift in which exactly $k$ associated sets contain 1 as an element, then $q_2(Z) = 2\binom{k}{2}$. This means that $X$ is not conjugate to any $\S$-limited shift.
\end{example}

A \textit{walk} on a set-equipped graph $(G,\mathcal S)$ is composed of a walk $W = (v_1,\dots,v_n)$ on $G$ together with a selection $(s_1,\dots,s_n)$ of an element $s_i \in S_{v_i}$ at each vertex. The walk is called \textit{closed} if $(v_n,v_1) \in E(G)$. The \textit{length} of the walk $W$ is $\ell(W) = n$, and its \textit{size}, denoted $\lvert W\rvert$, is the sum of the selected items; that is, $\lvert W\rvert = \sum_{i=1}^n s_i$.

We say that a word $\alpha \in \mathcal L\big(X(G,\mathcal S)\big)$ is comprised of \textit{full blocks} if $\alpha = u_1^{s_1}\cdots u_k^{s_k}$ for some $u_i \in V(G)$ with $s_i \in S_{u_i}$.

\begin{remark}
    Though they are equivalent, edge shifts and edge-labelled graphs are more common in the symbolic dynamics literature than vertex shifts and vertex-labelled graphs. This paper uses $\S$-graph shifts defined on vertices because unordered $\S$-limited shifts are most easily realized in this setting. However, all of the theory in this paper can be applied equally to $\S$-graph shifts arising from edges. (Indeed, the line graph of an edge-labelled graph is a vertex-labelled graph with the same associated $\S$-graph shift, to which the results of this paper may be applied.)
\end{remark}

\section{Dynamical properties}\label{sec:dynamical-properties}

This section provides conditions on a set-equipped graph that are necessary and sufficient to guarantee certain dynamical properties in the associated $\S$-graph shift. The proofs of these statements are somewhat technical and share no methods with the proofs of the main results in this paper. For this reason, the proofs have been consigned to an appendix, \cref{sec:appendix}. However, we will need the results in later sections, so this section states the theorems themselves.

Since the $S$-graph shift associated to a set-equipped graph remains unchanged after removing any nonessential vertices, we assume in this section and throughout the paper that every graph is essential. We first provide necessary and sufficient conditions for an $\mathcal{S}$-graph shift to be an SFT or sofic.

\begin{proposition}\label{thm:SFT}
The shift space $X(G,\mathcal{S})$ is of finite type if and only if every $S_v \in \mathcal S$ is either finite or cofinite.
\end{proposition}

With each subset $S\subseteq \N$, we can associate the (finite or infinite) sequence $(s_i)_{i=1}^{|S|}$ that enumerates the elements of $S$ in increasing order. If $S$ is finite with $n$ elements, then we extend $(s_i)_{i=1}^n$ to an infinite sequence by defining $s_i = s_{n}$ for any $i > n$. The \textit{difference sequence} of a set $S \subseteq \N$, denoted $\Delta(S)$, is the sequence of differences between consecutive members of $S$; that is, $\Delta(S) = (d_i)_{i=1}^\infty$ where $d_i = s_{i+1} - s_i$.

\begin{proposition}\label{thm:sofic}
The shift space $X(G,\mathcal{S})$ is sofic if and only if $\Delta(S_v)$ is eventually periodic for every $v \in V(G)$.
\end{proposition}

We next turn to the dynamical properties of mixing and (weak) specification.

\begin{definition}\label{def:mixing}
A shift space $X$ is called \textit{mixing} if for every pair of words $\alpha,\beta \in \mathcal L(X)$ the following is true: There exists a positive integer $N$ such that for every $n \geq N$ there is a word $\gamma \in \mathcal B_n(X)$ such that $\alpha\gamma\beta \in \mathcal L(X)$.
\end{definition}

\begin{proposition}\label{thm:mixing}
The shift space $X(G,\mathcal{S})$ is mixing if and only if $G$ is irreducible and \\ $\gcd\!\left\{\sum_{v\in C} s_v\, :\, \text{$C$ is a cycle in $G$ and $s_v \in S_v$}\right\}=1$.
\end{proposition}

\begin{definition}\label{def:specification}
A shift space $X$ has the \textit{specification property} if there exists a positive integer $N$ such that for every pair of words $\alpha,\beta \in \mathcal L(X)$ there is a $\gamma \in \mathcal B_N(X)$ with $\alpha\gamma\beta \in \mathcal L(X)$. The shift space $X$ has the \textit{weak specification property} if there exists a positive integer $N$ such that for every pair of words $\alpha, \beta \in \mathcal L(X)$ there is a word $\gamma \in \mathcal L(X)$ with $\lvert \gamma\rvert \leq N$ and $\alpha \gamma \beta \in \mathcal L(X)$.
\end{definition}

\begin{proposition}\label{thm:almost_spec}
The shift space $X(G,\mathcal{S})$ has the weak specification property if and only if $G$ is irreducible and $\sup \Delta(S_v) < \infty$ for each $v \in V(G)$.
\end{proposition}


\begin{proposition}\label{thm:specification}
An $\mathcal{S}$-graph shift has the specification property if and only if it is mixing and has the weak specification property.
\end{proposition}

\section{Proof of the entropy formula}\label{sec:entropy-formula}

The goal of this section is to prove \cref{thm:entropy_eigen=1}. As in the previous section, any graph in this section is essential. Moreover, we only consider nonempty $\mathcal{S}$-graph shifts, so every graph contains at least one cycle. For convenience, we set $V(G) = \{1,2,\dots,q\}$.

Spandl's theorem for the entropy of an $S$-gap shift associates the entropy of an $S$-gap shift to the roots of a power series:

\begin{theorem}\label{thm:S-gap-entropy}
    If $S \subseteq \N$ is nonempty and $x = \lambda$ is the unique positive solution to\vspace{-0.3ex}
    \[ \sum_{s \in S} x^{s+1} = 1,\vspace{-0.5ex} \]
    then the entropy of $X(S)$ is $\log \lambda^{-1}$.
\end{theorem}

The way we will prove \cref{thm:entropy_eigen=1} is similar in spirit, by relating the entropy of an $\S$-graph shift to the radius of convergence of a certain power series. By transitioning through several power series that have a stronger connection to the graph structure of the subshift and the same radius of convergence as the original series, we eventually transition to a purely linear algebraic statement that relates the entropy to the spectral radius of a matrix.

To do this, we will require a handful of combinatorial structures. Recall that the adjacency matrix of a graph $G$ is the $q \times q$ square matrix $A_G$ with $(A_G)_{i,j} = 1$ if $(i,j) \in E(G)$ and $(A_G)_{i,j} = 0$ otherwise.

\begin{definition}\label{def:generating-matrix}
Let $\mathcal G = (G,\mathcal S)$ be a set-equipped graph on the vertex set $\{1,2,\dots,q\}$. To each vertex $v \in V(G)$ we assign the generating function $H_v(x) = \sum_{s \in S_v} x^s$. The \textit{generating matrix} of $\mathcal G$ is the $q \times q$ matrix function $B_\mathcal G(x)$ with entries
\[  B_\mathcal G(x)_{i,j} = \begin{cases}
        H_i(x)  &\text{if } (i,j) \in E(G)\\
        0       &\text{otherwise.}
    \end{cases}\]
The \textit{diagonal generating matrix} is $D_\mathcal G(x) = \diag\!\big(H_1(x),\dots,H_n(x)\big)$. (So $B_\mathcal G(x) = D_\mathcal G(x)A_G$.)
\end{definition}

The next proposition provides the essential link between the graph representation of an $\mathcal{S}$-graph shift and its entropy. If $(G,\mathcal{S})$ is a set-equipped graph and $w = (v_1,\dots,v_n)$ is a walk on $G$, we let $H_w(x)$ denote the power series $\prod_{i=1}^n H_{v_i}(x)$. Recall that $W(G)$ denotes the set of all walks on $G$.

\begin{proposition}\label{thm:sum_over_walks}
Let $X=X(G,\mathcal{S})$ be an $\mathcal{S}$-graph shift. If $\lambda$ is the radius of convergence of the generating function $\sum_{w\in W(G)} H_w(x)$, then the entropy of $X$ is $\log \lambda^{-1}$.
\end{proposition}

Though the index set for this power series might seem somewhat out of control, it is indeed well-defined. If $w$ is a walk with $n$ vertices, then each term in $H_w(x)$ is divisible by $x^n$. In other words, the coefficient of $x^n$ in the power series $\sum_{w\in W(G)} H_w(x)$ depends only on the walks with $n$ or fewer vertices, of which there are finitely many. In fact, the coefficient of $x^k$ in $\sum_{w\in W(G)} H_w(x)$ is the number of walks $w$ whose size is $|w| = k$. One can show that the number of such walks is at most $q^k$, which implies that the radius of convergence of $\sum_{w\in W(G)} H_w(x)$ is at least $1/q$. Not coincidentally, this exactly coincides with the previous observation that the entropy of any shift space with $q$ letters is at most $\log q$.

\begin{proof}[Proof of \cref{thm:sum_over_walks}]
The power series $f(x) = \sum_{n=1}^{\infty}\mathcal{B}_n(X) x^n$ has radius of convergence $e^{-h(X)}$. In the rest of the proof, we show that the radius of convergence of $f(x)$ is $\lambda$, which implies that $h(X)=\log\lambda^{-1}$.

Since every coefficient in both series is positive and the series are centered at 0, the radius of convergence can be determined by examining convergence only at nonnegative real numbers. Let $W = W(G)$. The coefficient of $x^n$ in $\sum_{w\in W} H_w(x)$ is the number of words of the form $u_1^{s_1}\cdots u_j^{s_j}$ where $s_i \in S_{u_i}$, each $(u_i,u_{i+1}) \in E(G)$, and $s_1 + \cdots + s_j = n$. The number of such words is at most $\mathcal B_n(X)$, so
\[ \sum_{w\in W} H_w(x) \leq f(x) \]
for every $x \geq 0$. Every allowable word has the form $a^{k_1}\omega b^{k_2}$, where $\omega$ is composed of concatenated full blocks, $a,b \in V(G)$, and $k_1,k_2 \in \mathbb{N}$. The number of such words (ignoring for the moment whether they are in $\mathcal L(X)$) has generating function
\[ \bigg(\frac{q}{1-x}\bigg) \bigg(\sum_{w\in W} H_w(x)\bigg) \bigg(\frac{q}{1-x}\bigg). \]
Each coefficient of this power series is at least as large as the corresponding coefficient of $f(x)$, so\vspace{-0.5ex}
\[ \sum_{w\in W} H_w(x) \leq f(x) \leq \left(\frac{q}{1-x}\right)^{\!2}\sum_{w\in W} H_w(x) \]
for every $x \geq 0$. Because $G$ contains at least one cycle, $\sum_{w\in W} H_w(x)$ contains arbitrarily large powers of $x$; all coefficients are nonnegative integers, so the power series diverges at $x=1$. The power series $(1-x)^{-1}$ has a radius of convergence of $1$, so $f(x)$ will converge exactly when $\sum_{w\in W} H_w(x)$ does. Therefore $e^{-h(X)}=\lambda$.
\end{proof}

\cref{thm:sum_over_walks} is enough to calculate the entropy of some simple $\mathcal{S}$-graph shifts, in particular when each set $S_i$ is the same or $G$ is quite small, but the calculation rapidly becomes intractable in any more complicated situation. To tame the sum $\sum_{w \in W(G)} H_w(x)$, we'll need finer control over the vertices in a walk.

\begin{definition}
Let $G$ be a graph with vertices $v_1, \dots, v_q$, and let $A_G$ denote the adjacency matrix of $G$. Considering each of the vertices as indeterminates, we define $D_G = \diag(v_1,\dots,v_q)$. The \textit{symbolic adjacency matrix} of $G$ is $B_G = D_GA_G$.
\end{definition}

We may omit the subscript if only one graph is under discussion. The symbolic adjacency matrix is the labelled adjacency matrix for the graph obtained from $G$ by labelling each edge with the vertex it is directed from.
\begin{example}
    A graph and its symbolic adjacency matrix.
    \begin{center}
        \hspace{0.08\textwidth}
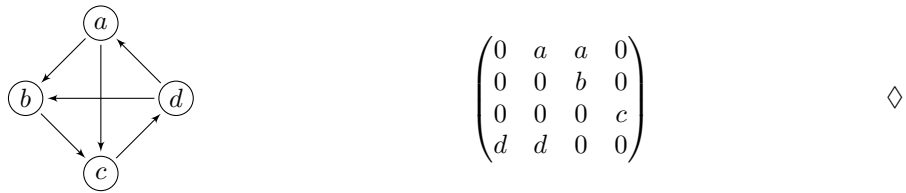
        \begin{minipage}{0.4\textwidth}
        \begin{center}
        \begin{tikzpicture}
            \foreach \a/\b/\c/\d in {0/1/4/d, 90/1/1/a, 180/1/2/b, -90/1/3/c}
                \node (\c) at (\a:\b) {$\d$};
            \foreach \tail/\head in {1/2,2/3,3/4,4/1, 1/3, 4/2}
                \draw[edge] (\tail) to (\head);
        \end{tikzpicture}
        \end{center}
        \end{minipage}
        \begin{minipage}{0.4\textwidth}
            \[ \begin{pmatrix}
            0 & a & a & 0\\
            0 & 0 & b & 0\\
            0 & 0 & 0 & c\\
            d & d & 0 & 0
            \end{pmatrix} \]
        \end{minipage}\qedhere
    \end{center}
\end{example}

\begin{lemma}\label{thm:symbolic_adj}
Let $G$ be a graph with $D_G$ and $B_G$ as above. The number of walks of length $n$ on $G$ which, for each $1 \leq i \leq q$, visit vertex $v_i$ exactly $k_i$ times is the coefficient of $v_1^{k_1}\cdots v_q^{k_q}$ in $\sum_{i,j} (B^n D)_{i,j}$.
\end{lemma}
\begin{proof}
Induction on $n$ shows that the number of walks that visit $v_i$ exactly $k_i$ times and begin at $v_s$ is the coefficient of $v_1^{k_1}\cdots v_q^{k_q}$ in $\sum_{j=1}^q (B^nD)_{s,j}$. Summing over all values of $s$ proves the lemma. 
\end{proof}

\begin{corollary}
If $\mathcal G = (G,\mathcal S)$ is a set-equipped graph, then
\[ \sum_{n=0}^\infty \sum_{i,j} \big(B_{\mathcal G}(x)^nD_{\mathcal G}(x)\big)_{i,j} =\ \sum_{\mathclap{w\in W(G)}} \,H_w(x). \]
\end{corollary}
\begin{proof}
The coefficient of $H_1(x)^{k_1}\cdots H_q(x)^{k_q}$ in the sum $\sum_{w\in W(G)} H_w(x)$ is the number of walks in $G$ that visit vertex $i$ exactly $k_i$ times, so it is equal to the coefficient of $v_1^{k_1}\cdots v_q^{k_q}$ in $\sum_{i,j}(B_G^n D)_{i,j}$. Summing over all compositions $k_1 + \cdots + k_q$ of $n$ and setting $v_i = x$ for every $1 \leq i \leq q$ shows that the coefficients of $x^n$ in $\sum_{w\in W(G)} H_w(x)$ and $\sum_{n=0}^\infty \sum_{i,j} \!\big(B_{\mathcal G}(x)^nD_{\mathcal G}(x)\big)_{i,j}$ are equal.
\end{proof}

We can eliminate the diagonal matrix while maintaining the radius of convergence.

\begin{lemma}\label{thm:matrix_convergence}
The power series $\sum_{n=0}^\infty \sum_{i,j} \!\big(B_{\mathcal G}(x)^n D_{\mathcal G}(x)\big)_{i,j}$ and $\sum_{n=0}^\infty \sum_{i,j} \!\big(B_{\mathcal G}(x)^n\big)_{i,j}$ have the same radius of convergence.
\end{lemma}
\begin{proof}
    As in the proof of \cref{thm:sum_over_walks}, the radii of convergence can be determined by examining only nonnegative numbers. Fix some $x \in \mathbb{R}_{\geq 0}$. If $H_v(x)$ diverges for any $v \in V(G)$, then both $\sum_{n=0}^\infty \sum_{i,j} \big(B_{\mathcal G}(x)^n\big)_{i,j}$ and $\sum_{n=0}^\infty \sum_{i,j} \big(B_{\mathcal G}(x)^n D_{\mathcal G}(x)\big)_{i,j}$ diverge. So suppose that $H_v(x)$ is finite for each $v \in V(G)$. If $c$ is the minimum of $\{H_1(x), \dots, H_q(x)\}$ and $d$ is the maximum, then
    \[ c \sum_{n=0}^\infty \sum_{i,j} \!\big(B_{\mathcal G}(x)^n\big)_{i,j} \leq
    \sum_{n=0}^\infty \sum_{i,j} \!\big(B_{\mathcal G}(x)^n D_{\mathcal G}(x)\big)_{i,j} \leq
    d \sum_{n=0}^\infty \sum_{i,j} \!\big(B_{\mathcal G}(x)^n\big)_{i,j}. \]
    Since both $c$ and $d$ are nonnegative, it follows that $\sum_{n=0}^\infty \sum_{i,j} \big(B_{\mathcal G}(x)^nD_{\mathcal G}(x)\big)_{i,j}$ converges if and only if $\sum_{n=0}^\infty \sum_{i,j} \big(B_{\mathcal G}(x)^n\big)_{i,j}$ does.
\end{proof}

To complete the proof of \cref{thm:entropy_eigen=1}, we need two more ingredients. The first is a consequence of the Perron-Frobenius theorem from linear algebra (see Chapter 4 of \cite{Lind-Marcus}). We call a nonnegative square matrix $A$ \textit{irreducible} if for every pair of indices $i,j$, there is a positive integer $n$ so that $(A^n)_{i,j} > 0$. The adjacency matrix $A_G$ is irreducible if and only if $G$ is irreducible as a graph (as defined on page \pageref{def:irreducible-graph}).

\begin{theorem}\label{thm:perron-frobenius-original}
If $A$ is an irreducible matrix with spectral radius $\rho$, then there exist positive real numbers $c$ and $d$ such that
\[  c\rho^n \leq \sum_{i,j} (A^n)_{i,j} \leq d\rho^n\vspace{-0.75\baselineskip}  \]
for every $n \in \N$.
\end{theorem}

This result can be extended to all nonnegative matrices with only a modest change to the upper bound. The proof of this proposition is somewhat \textit{ad hoc} (in the context of this paper), so it has been relegated to \cref{sec:appendix}.

\begin{proposition}\label{thm:perron-frobenius}
If $A$ is a nonnegative square matrix with spectral radius $\rho$, then there exist positive real numbers $c$ and $d$ and an integer $k$ such that
\[  c\rho^n \leq \sum_{i,j} (A^n)_{i,j} \leq dn^k\rho^n\vspace{-0.5\baselineskip}  \]
for every $n \in \N$.
\end{proposition}

The second ingredient is a lemma about the spectral radius of a generating matrix; its proof, too, may be found in \cref{sec:appendix}.

\begin{lemma}\label{thm:lambda_is_increasing}
If $(G,\mathcal S)$ is a set-equipped graph and $\rho(x)$ is the spectral radius of $B_{\mathcal G}(x)$, then $\rho$ is continuous and strictly increasing on the set of nonnegative real numbers with range $[0,+\infty)$.
\end{lemma}

We can now prove our first main theorem, stated in terms of set-equipped graphs below.

\begin{reptheorem}{thm:entropy_eigen=1}
    Let $\mathcal G = (G,\mathcal{S})$ be a set-equipped graph and $\rho(x)$ be the spectral radius of $B_{\mathcal G}(x)$. The equation $\rho(x)=1$ has a unique positive solution $x=\lambda$, and the entropy of $X(G,\mathcal{S})$ is $\log \lambda^{-1}$.
\end{reptheorem}
\begin{proof}
Let $x > 0$ be within the radius of convergence of each $H_i(x)$. By \cref{thm:perron-frobenius} there are positive constants $c$, $d$, and $k$ such that
\[ c\rho(x)^n \leq \sum_{i,j} \,\big(B_{\mathcal G}(x)^n\big)_{i,j} \leq dn^k \rho(x)^n \]
for every $n \in \N$. Summing over $n$, we get
\[ c \sum_{n=0}^\infty \rho(x)^n \leq
    \sum_{n=0}^\infty \sum_{i,j} \,\big(B_{\mathcal G}(x)^n\big)_{i,j} \leq
    d \sum_{n=0}^\infty n^k\rho(x)^n. \]
The bounding series converge if and only if $\vert \rho(x) \vert < 1$. \cref{thm:lambda_is_increasing} guarantees a unique positive solution $x = \lambda$ to the equation $\rho(x) = 1$. Since $\rho$ is a strictly increasing function, the series $\sum_{n=0}^\infty \sum_{i,j} (B_{\mathcal G}(x)^n)_{i,j}$ converges when $x < \lambda$ and diverges when $x > \lambda$, so $\lambda$ is the radius of convergence. By \cref{thm:sum_over_walks}, the entropy of $X(G,\mathcal{S})$ is $\log \lambda^{-1}$.
\end{proof}

\begin{remark}
    In practice, the solution to $\rho(x) = 1$ can be calculated by determining the least positive value of $x$ for which $\det\!\big(I - B_\mathcal G(x)\big) = 0$. The reason is that $\det\!\big(I - B_\mathcal G(x)\big) = 0$ exactly when 1 is an eigenvalue of $B_\mathcal G(x)$. This happens only when $\rho(x) \geq 1$ and certainly when $\rho(x) = 1$. Since $\rho$ is an increasing function, the least positive solution to $\rho(x) = 1$ is also the least solution to $\det\!\big(I - B_\mathcal G(x)\big) = 0$.
\end{remark}

A shift space $X$ is called \textit{almost sofic} if for every $\epsilon > 0$ there exists a subshift of finite type $Y \subseteq X$ such that $h(Y) \geq h(X) - \epsilon$. Almost sofic shifts were introduced by Petersen in \cite{almost-sofic}; they are the class of shift spaces that are particularly interesting from the perspective of data encoding.

\begin{corollary}
Every $\mathcal{S}$-graph shift is almost sofic.
\end{corollary}
\begin{proof}
Let $\mathcal G = (G,\mathcal S)$ be a set-equipped graph. For each $v \in V(G)$ and $n \in \N$, define $S_v^n = S_v \cap \{1,2,\dots,n\}$ and $\mathcal S_n = (S^n_v)_{v \in V(G)}$ and set $\mathcal G_n = (G,\mathcal S_n)$. Then $B_{\mathcal G_n}(x)$, considered as a matrix-valued function of $x$, tends to $B_{\mathcal G}(x)$ pointwise as $n \to \infty$. If $\rho_n(x)$ as the spectral radius of $B_{\mathcal G_n}(x)$ and $\rho(x)$ as the spectral radius of $B_\mathcal G(x)$, then $\rho_n(x) \to \rho(x)$ pointwise as $n \to \infty$ (where $\rho_n(x)$ and $\rho(x)$ are considered as real-valued functions of $x$). \cref{thm:lambda_is_increasing} and \cref{thm:entropy_eigen=1} together imply that $h(X(\mathcal G_n)) \to h(X(\mathcal G))$. Since every set in $\S_n$ is finite, \cref{thm:SFT} guarantees that $X(\mathcal G_n)$ is an SFT for every $n \in \N$, so $X(\mathcal G)$ is almost sofic.
\end{proof}

\cref{thm:entropy_eigen=1} can be phrased in terms of an explicit power series to more closely mirror the entropy formula for $S$-gap shifts. We denote by $\mathcal C_G$ the collection of all nonempty sets of vertex-disjoint cycles in a graph $G$. That is, each element of $\mathcal C_G$ is itself a set of cycles in $G$ which do not overlap. For $C \in \mathcal C_G$, we let $|C|$ denote the cardinality of $C$. Recall that the length of a cycle $c$, denoted $\ell(c)$, is the number of vertices that appear in the cycle.

\begin{theorem}\label{thm:entropy_sumovercycles}
Let $(G,\mathcal S)$ be a set-equipped graph. If $\lambda$ is the least positive solution to
\begin{equation}\label{eq:power-series-entropy-formula}
    \sum_{C \in \mathcal C_G}(-1)^{|C| + 1} \prod_{c\,\in C} H_c(x) = 1,
\end{equation}
then the entropy of $X(G,\mathcal{S})$ is $\log\lambda^{-1}$.
\end{theorem}
\begin{proof}
We know that the least positive solution to $\rho(x)=1$ is also the least positive solution to $\det(B_{\mathcal G}(x)-I)=0$. We will therefore show that $\det(B_{\mathcal G}(x) - I) = 0$ if and only if \eqref{eq:power-series-entropy-formula} holds; appealing to \cref{thm:entropy_eigen=1} finishes the proof.

Let $\pi$ be a permutation of $V(G)$. A nontrivial cycle in the cycle decomposition of $\pi$ corresponds to a cycle in $G$ if and only if $\big(i,\pi(i)\big) \in E(G)$ for each element $i$ in that cycle. But $\big(i,\pi(i)\big) \in E(G)$ exactly when $\!\big(B_{\mathcal G}(x)\big)_{i,\pi(i)} \not= 0$ for every $x > 0$. So either the nontrivial cycles of $\pi$ correspond to a set of disjoint cycles in $G$, or the corresponding term in the permutation expansion of $\det\!\big(B_{\mathcal G}(x) - I\big)$ is 0. Suppose $\pi$ corresponds to the set $C=\{c_1,\dots,c_m\}$ of disjoint cycles in $G$. The sign of $\pi$ is $\sgn(\pi) = \prod_{i=1}^m (-1)^{\ell(c_i)+1}$. Also, we have
\[ (B_{\mathcal G}(x)-I)_{i,\pi(i)}=\begin{cases}
H_i(x) &\text{if } i\not=\pi(i)\\
-1 &\text{if } i=\pi(i).
\end{cases} \]
If $\pi$ fixes exactly $s$ vertices, then (recalling that $\lvert V(G)\rvert = q$), we have $\ell(c_1) + \cdots + \ell(c_m) + s = q$ and
\[ \sgn(\pi) \prod_{i=1}^q B_{\mathcal G}(x)_{i,\pi(i)} = \left[\prod_{i=1}^m (-1)^{\ell(c_i)+1}\right] \left[(-1)^s \prod_{c \in C} H_c(x)\right] = (-1)^{m+q} \prod_{c \in C} H_c(x). \]
Thus
\begin{align*}
\det(B_{\mathcal G}(x)-I) &= \sum_{\pi \in \mathfrak{S}_q} \sgn(\pi) \prod_{i=1}^q B_{\mathcal G}(x)_{i,\pi(i)}\\
&= (-1)^q + \sum_{C \in \mathcal C_G} (-1)^{\vert C\vert+q} \prod_{c\in C} H_c(x).
\end{align*}
Therefore $\det(B_{\mathcal G}(x)-I)=0$ if and only if
\[ \sum_{C \in \mathcal C_G} (-1)^{\vert C\vert+1} \prod_{c\in C} H_c(x) = 1. \qedhere\]
\end{proof}

Equation \eqref{eq:power-series-entropy-formula} generalizes the entropy formula for ordered $\mathcal{S}$-limited shifts. The ordered $\mathcal{S}$-limited shift $\vv{X}(\mathcal S)$ can be represented as the $\mathcal{S}$-graph shift $X(G,\mathcal S)$ where $G$ is a directed cycle. Then $\mathcal C_G$ contains only a single cycle (the whole graph), and the power series in \eqref{eq:power-series-entropy-formula} reduces to
\[  \prod_{i=1}^q H_i(x) = \sum_{\omega \in G_{\mathcal{S}}} x^{\vert \omega\vert},  \]
which is the power series that appears in Matson and Sattler's entropy formula for ordered $\S$-limited shifts \cite[Theorem 4.1]{S-limited}.

\section{Zero-entropy shifts}\label{sec:graph-operations}

The generating matrix of a set-equipped graph provides a link between the graph structure and the dynamics of an $\S$-graph shift. In this section, we leverage that link to characterize those $\S$-graph shifts with entropy 0.

\begin{definition}
    Let $(G,\mathcal{S})$ be a set-equipped graph, $v$ be a vertex of $G$, and suppose that $S_v = S_1 \sqcup S_2$. The \textit{vertex cloning} of $G$ at $v$ with respect to the decomposition $S_1 \sqcup S_2$ the set-equipped graph $(G',\mathcal{S}')$ defined as follows: The vertex set of $G'$ is $V(G) \cup \{v'\}$, and $(u_1,u_2) \in E(G')$ if and only if one of the following is true:
    \begin{itemize}
        \item $u_1,u_2 \notin \{v,v'\}$ and $(u_1,u_2) \in E(G)$,
        \item $u_1 \in \{v,v'\}$ and $(v,u_2) \in E(G)$, or
        \item $u_2 \in \{v,v'\}$ and $(u_1,v) \in E(G)$.
    \end{itemize}
    Finally, $S_u' = S_u$ if $u \notin \{v,v'\}$, $S'_v = S_1$, and $S_{v'}' = S_2$.
\end{definition}

\begin{example}\label{ex:vertex-cloning}
The graph on the right is a vertex cloning of the graph on the left.
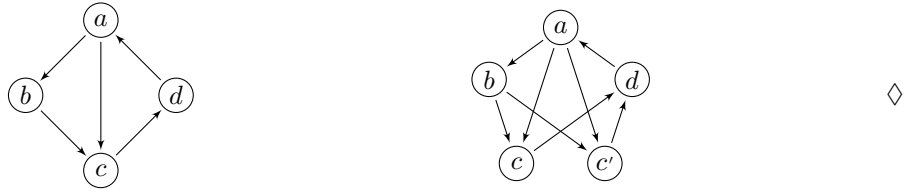
\begin{center}
    \hspace{0.08\textwidth}
    \begin{minipage}{0.4\textwidth}
        \begin{center}
        \begin{tikzpicture}[shorten > = 1pt]
            \tikzset{edge/.style = {->,> = latex'}}
            \node (a) at (0,1) {$a$};
            \node (b) at (-1,0) {$b$};
            \node (c) at (0,-1) {$c$};
            \node (d) at (1,0) {$d$};
            
            \foreach \tail/\head in {a/b,a/c,b/c,c/d,d/a}
                \draw[edge] (\tail) to (\head);
        \end{tikzpicture}
        \end{center}
    \end{minipage}
    \begin{minipage}{0.4\textwidth}
        \begin{center}
        \begin{tikzpicture}
            \node (a) at (90:1) {$a$};
            \node (b) at (162:1) {$b$};
            \node (c1) at (234:1) {$c$};
            \node (c2) at (306:1) {$c\mspace{1mu}\scalebox{0.7}{$'$}$};
            \node (d) at (18:1) {$d$};
            
            \foreach \tail/\head in {a/b,a/c1,a/c2,b/c1,b/c2,c1/d,c2/d,d/a}
                \draw[edge] (\tail) to (\head);
        \end{tikzpicture}
        \end{center}
    \end{minipage}\qedhere
\end{center}
\end{example}
\vspace{0.5ex}

To capture the effect of vertex cloning on entropy, we'll use the following lemma.

\begin{lemma}\label{thm:spectral_matrix_lemma}
Let $A$ be a $q\times q$ matrix with rows $a_i$ and suppose $a_m = x+y$. If $A'$ is the matrix obtained by duplicating the $m$th column in the matrix whose rows are $a_1,\dots,a_{m-1},x,y$, $a_{m+1},\dots,a_q$, then $\spec(A')=\spec(A)\cup\{0\}$.
\end{lemma}
\begin{proof}
We calculate the characteristic polynomial of $A'$. By interchanging rows and columns, we can assume that $m=q$, so that $A'-\rho I$ has the form
\[ \begin{pmatrix}
a_{1,1}-\rho & a_{1,2} & \cdots & a_{1,q} & a_{1,q}\\
a_{2,1} & a_{2,2}-\rho & \cdots & a_{2,q} & a_{2,q}\\
\vdots & \vdots & \ddots & \vdots & \vdots\\
b_1 a_{q,1} & b_2 a_{q,2} & \cdots & b_q a_{q,q}-\rho & b_q a_{q,q}\\
c_1 a_{q,1} & c_2 a_{q,2} & \cdots & c_q a_{q,q} & c_q a_{q,q}-\rho
\end{pmatrix}, \]
where $b_i + c_i = 1$ for all $1\leq i \leq q$. By adding the last row to the penultimate one and then subtracting the penultimate column from the final one, we obtain the matrix
\[ \begin{pmatrix}
a_{1,1}-\rho & a_{1,2} & \cdots & a_{1,q} & 0\\
a_{2,1} & a_{2,2}-\rho & \cdots & a_{2,q} & 0\\
\vdots & \vdots & \ddots & \vdots & \vdots\\
a_{q,1} & a_{q,2} & \cdots & a_{q,q}-\rho & 0\\
c_1 a_{q,1} & c_2 a_{q,2} & \cdots & c_q a_{q,q} & -\rho
\end{pmatrix}. \]
Therefore $\det(A'-\rho I) = -\rho \det(A-\rho I)$, and the conclusion follows.
\end{proof}

\begin{proposition}\label{thm:cloning_preserves_entropy}
Vertex cloning preserves entropy and the mixing property of the associated $\S$-graph shift.
\end{proposition}
\begin{proof}
Suppose that $Y = X(\mathcal G')$ is obtained from the shift space $X = X(\mathcal G)$ by cloning vertex $v$. We can take $B_{\mathcal G}(x)$ and $B_{\mathcal G'}(x)$ to be $A$ and $A'$, respectively, in \cref{thm:spectral_matrix_lemma}. Therefore their spectral radii are equal for every $x$, so \cref{thm:entropy_eigen=1} implies that $h(X)=h(Y)$.

Let $U(\mathcal G) = \left\{\sum_{v\in C} s_v\, :\, \text{$C$ is a cycle in $G$ and $s_v \in S_v$}\right\}$. We have $U(\mathcal G) \subseteq U(\mathcal G')$. By \cref{thm:mixing}, if $X$ is mixing, then $Y$ is, as well.
\end{proof}

This is already enough to characterize zero-entropy $\S$-graph shifts.

\begin{proposition}\label{thm:zero_entropy_shifts}
An irreducible $\mathcal{S}$-graph shift $X(G,\mathcal{S})$ has entropy 0 if and only if $G$ is a directed cycle and each set in $\mathcal{S}$ is a singleton.
\end{proposition}
\begin{proof}
If $G$ is not a directed cycle, then some vertex $v$ of $G$ has outdegree at least 2. Set
\[ N = \max \Big\{ \sum_{u \in c} \min S_u :
        c \text{ is a cycle that includes } v\Big\}.\vspace{-0.5ex} \]
We can build words in $\mathcal L(X)$ by beginning a walk at $v$, choosing an edge on which to leave, circling back to $v$, and repeating. Since returning to $v$ takes at most $N$ letters, there are at least $2^{\lceil n/N\rceil}$ walks of length $n$ in $G$. So $h(X) \geq (\log 2) / N$.

Now suppose that $G$ is a directed cycle. If some set $S_v$ is not a singleton, then choose some nonempty sets $A,B \subseteq \N$ such that $S_v = A \sqcup B$. By \cref{thm:cloning_preserves_entropy}, the vertex cloning of $X$ at $v$ preserves entropy. Since the resulting graph is not a directed cycle, $h(X) > 0$.

Finally, suppose that $G$ is a directed cycle with $q$ vertices and each set in $\mathcal{S}$ is a singleton. Suppose $S_i = \{s_i\}$ and set $s = \sum_{i=1}^q s_i$. The only positive solution to $x^s = 1$ is $x=1$, so the entropy of $X$ is $\log 1^{-1} = 0$ by \cref{thm:entropy_sumovercycles}.
\end{proof}

As a consequence, if $X(\mathcal G)$ is an $\S$-graph shift with entropy 0, then every irreducible component of $\mathcal G$ is a cycle, and each set is a singleton. In other words, $X(\mathcal G)$ consists of the orbits of finitely many periodic points.

A few more comments are in order on the effect of this graph operation on the resulting shift space. Even though vertex cloning preserves entropy and mixing, it is not a conjugacy in general: It does not always preserve weak or strong specification, SFTs, or sofic shifts. It can also change the number of points of period 1: If in \cref{ex:vertex-cloning}, for example, $S_v = \N$ for every $v \in \{a,b,c,d\}$ and we set $S_1 = 2\N$ and $S_2 = 2\N-1$, then the original $\S$-graph shift has 4 points with period 1, while the resulting shift has 5.

On the other hand, splitting into two infinite sets is the only thing that prevents a conjugacy.

\begin{proposition}\label{thm:cloning_conjugacy}
Let $(G,\mathcal{S})$ be a set-equipped graph and $v \in V(G)$. If $S_v = S_1 \sqcup S_2$ and either $S_1$ or $S_2$ is finite, then the shift space obtained by cloning vertex $v$ is conjugate to $X(G,\mathcal{S})$.
\end{proposition}
\begin{proof}
Let $X = X(G,\mathcal S)$ and $Y$ be the shift space obtained by vertex cloning at $v$ with respect to the partition $S_1 \sqcup S_2$. By symmetry, we assume that $S_1$ is finite and let $N$ be its maximum. Let $\Phi\colon \mathcal{B}_{2N+1}(X) \to V(G')$ be defined by
\[  \Phi(x_{-N}\cdots x_N) = \begin{cases}
        x_0 &\text{if } x_0 \neq v\\
        v   &\text{if } x_{-s},\dots,x_{t} = v,\ x_{-s-1} \neq v,\ x_{t+1}\neq v \text{, and } t+s+1 \in S_1\\
        v' &\text{if } x_{-s},\dots,x_{t} = v,\ x_{-s-1} \neq v,\ x_{t+1}\neq v \text{, and } t+s+1 \in S_2.
    \end{cases} \]
The map $\Phi$ induces a sliding block code $\phi\colon X\to Y$, whose inverse is induced by the map $\Psi\colon \mathcal{B}_1(Y) \to V(G)$ defined by
\[ \Psi(y_0) = \begin{cases}
    y_0 &\text{if } y_0 \neq v'\\
    v   &\text{if } y_0 = v'.
\end{cases} \]
So $X \cong Y$.
\end{proof}

\section{Shifts with a given entropy}\label{sec:entropy-construction}

In this section, we answer the question: Does there exist an $\mathcal{S}$-graph shift with entropy $h$ for every $h \geq 0$? \cref{thm:shift_of_every_entropy} answers this question in the affirmative, and \cref{thm:every-entropy}, which we prove at the end of this section, strengthens this statement considerably. To prove these statements, we'll need to dip into a bit of number theory.

\begin{definition}
    Let $\beta > 1$ be a real number. A \textbf{$\beta$-expansion} for the real number $x$ is a sequence $(x_n)_{n=1}^\infty$ such that each $x_n \in \{0,1,\dots,\lceil \beta\rceil - 1\}$ and
    \[ x = \frac{x_1}{\beta} + \frac{x_2}{\beta^2} + \frac{x_3}{\beta^3} + \cdots. \]
\end{definition}

R\'enyi proved, in a much more general setting, that every real number in the interval $[0,(\lceil \beta\rceil-1)/(\beta-1)]$ has at least one $\beta$-expansion \cite{renyi-beta-expansions}. (This can also be proven directly using a greedy algorithm.)

In \cite{dynamical-properties}, Baker and Ghenciu use $\beta$-expansions to prove that for every $1 \leq \lambda \leq 2$, there is an $S$-gap shift with entropy $\log \lambda$. We now extend this result to $\S$-graph shifts.

\begin{proposition}\label{thm:shift_of_every_entropy}
    For every real number $\lambda > 1$, there exists an $\mathcal{S}$-graph shift with entropy $\log \lambda$.
\end{proposition}

In the following proof, we will use the directed complete bipartite graph $K_{n,n}$, which has vertex set $\{1,2,\dots,2n\}$. The vertices are partitioned into two sets $V_1 = \{1,2,\dots,n\}$ and $V_2 = \{n+1,n+2,\dots,2n\}$, and $(u,v)$ is an edge of $K_{n,n}$ if and only if either
\begin{itemize}
    \item $u \in V_1$ and $v \in V_2$ or
    \item $u \in V_2$ and $v \in V_1$.
\end{itemize}
A drawing of $K_{3,3}$ appears in \cref{ex:every-entropy}.

\begin{proof}[Proof of \cref{thm:shift_of_every_entropy}]
    Let $(x_i)_{i=1}^\infty$ be a $\lambda$-expansion of $1$ and set $n=\lceil \lambda\rceil-1$. For each $1 \leq k \leq n$, we define the set
    \[ T_k = \{ m \in \N : x_m \geq k\}. \]
    Let $G$ be the complete bipartite graph $K_{n,n}$ with vertex bipartition $[\{1,2,\dots, n\},\{n+1,\dots,2n\}]$. For each $1\leq k \leq n$ we define
    \[ S_k = S_{k+n} = T_k\]
    and set $\mathcal{S} = (S_k)_{k=1}^{2n}$. We claim that $X(G,\mathcal{S})$ has entropy $\log \lambda$.
    
    By \cref{thm:entropy_eigen=1}, it suffices to show that the maximal eigenvalue of $B_{\mathcal G}(\lambda^{-1})$ is 1. Using $0_{m\times n}$ and $1_{m\times n}$ to denote the $m\times n$ constant matrices filled with 0's and 1's, respectively, we have
    \[
        B_{\mathcal G}(x) = 
        \begin{pmatrix}
            H_1(x) &&\\
            & \ddots&\\
            && H_{2n}(x)
        \end{pmatrix}
        \begin{pmatrix}
            0_{n\times n} & 1_{n\times n}\\
            1_{n\times n} & 0_{n\times n}
        \end{pmatrix}.
    \]
    For each $1 \leq i \leq 2n$, let $e_i$ denote the $i$th standard basis vector in $\mathbb{R}^{2n}$. If $1\leq k \leq n$, then $e_k - e_{k+1}$ and $e_{k+n} - e_{k+n+1}$ are eigenvectors of $B_{\mathcal G}(x)$ with eigenvalue 0. Moreover, using the fact that $H_k(x) = H_{k+n}(x)$ for each $1 \leq k \leq n$, we have that $\sum_{k=1}^{2n} H_k(x)\,e_k$ is an eigenvector with eigenvalue $\sum_{k=1}^n H_k(x)$, and $\sum_{k=1}^n \big(H_k(x)\,e_k - H_{k+n}(x)\,e_{k+n}\big)$ is an eigenvector with eigenvalue $-\sum_{k=1}^n H_k(x)$. This is a list of $2n$ linearly independent eigenvectors; from this eigenbasis, we can easily see that the spectral radius of $B_{\mathcal G}(x)$ is $\sum_{k=1}^n H_k(x)$. So the spectral radius of $B_{\mathcal G}(\lambda^{-1})$ is 
    \[ \sum_{k=1}^n H_{k}(\lambda^{-1}) = \sum_{i=1}^\infty \frac{x_i}{\lambda^i} = 1. \]
    By \cref{thm:entropy_eigen=1}, the entropy of $X(G,\mathcal S)$ is $\log \lambda$.
\end{proof}

\begin{example}\label{ex:every-entropy}
    To construct an $\mathcal{S}$-graph shift with entropy $\log(e+1)$, we first note that $\lceil e+1 \rceil - 1 = 3$. Then take any $(e+1)$-expansion of $1$; one such expansion begins
    \[
        (3,2,2,1,3,0,1,\dots).
    \]
    The sets corresponding to this expansion are 
    \begin{align*}
        T_1 &= \{1,2,3,4,5,7,\dots\}\\
        T_2 &= \{1,2,3,5,\dots\}\\
        T_3 &= \{1,5,\dots\}.
    \end{align*}
    We set $G$ to be the graph
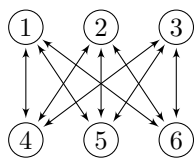
\vspace{-0.4\baselineskip}
    \begin{center}
    \begin{tikzpicture}[scale=1]
        \tikzset{shorten < = 1pt}
        \node (a) at (0,1.5) {1};
        \node (b) at (1,1.5) {2};
        \node (c) at (2,1.5) {3};
        \node (d) at (0,0) {4};
        \node (e) at (1,0) {5};
        \node (f) at (2,0) {6};
        
        \foreach \tail/\head in {a/d,a/e,a/f,b/d,b/e,b/f,c/d,c/e,c/f}
            \draw[edge, <->] (\tail) to (\head);
    \end{tikzpicture}
    \end{center}
    and define $\mathcal S = (S_k)_{k=1}^6$ by $S_1 = S_4 = T_1$,\, $S_2 = S_5 = T_2$, and $S_3 = S_6 = T_3$. Then $h\big(X(G,\mathcal S)\big) = \log(e+1)$.
\end{example}

We will use two lemmas---one on $\beta$-expansions and one on general subshifts---to prove \cref{thm:every-entropy}.

\begin{lemma}\label{thm:mu_exp_of_1}
For every $\beta>1$, there is a $\beta$-expansion of 1 that contains a pair of consecutive nonzero digits and in which eventually every other digit is $\lceil \beta\rceil-1$.
\end{lemma}

See \cref{sec:appendix} for a proof of \cref{thm:mu_exp_of_1}.

\begin{lemma}\label{thm:countable-conjugates}
Let $X$ be a shift space on a finite alphabet. There are only countably many shift spaces on the same alphabet that are conjugate to $X$.
\end{lemma}
\begin{proof}
Suppose the alphabet of $X$ is $\mathcal A$. Each conjugate of $X$ is determined by a sliding block code, which is itself induced by a positive integer $n$ and a map $\mathcal A^n \to \mathcal A$, and there are only countably many such maps.
\end{proof}

\begin{reptheorem}{thm:every-entropy}
Let $\lambda > 1$. For each $q \geq 2\lceil \lambda \rceil - 1$, there are uncountably many pairwise non-conjugate $\mathcal{S}$-graph shifts on $q$ vertices that have the specification property and entropy $\log \lambda$.
\end{reptheorem}

\begin{proof}
Let $n=\lceil \lambda \rceil -1$. By \cref{thm:mu_exp_of_1}, there is a $\lambda$-expansion $(x_i)_{i=1}^\infty$ of $1$ that contains a pair of consecutive nonzero integers and in which eventually every other digit is $n$. Let $X = X(G,\mathcal S)$ be the $\mathcal{S}$-graph shift constructed in the proof of \cref{thm:shift_of_every_entropy} that corresponds to the $\lambda$-expansion $(x_i)_{i=1}^\infty$. The entropy of $X$ is $\log \lambda$, and $G$ has $2n$ vertices.

Since eventually every other term of $(x_i)$ is $n$, it follows that, for every $k \in \{1,2,\dots,2n\}$, eventually every other positive integer is in $S_k$. Therefore $\sup \Delta(S_k) < \infty$; by \cref{thm:almost_spec}, $X$ has weak specification. If $x_m$ and $x_{m+1}$ are a pair of consecutive nonzero terms of $(x_i)$, then $m$ and $m+1$ are both elements of $S_1 = S_{n+1}$. The vertices $1$ and $n+1$ comprise a length 2 cycle in $G$, so
\[ 2m,\, 2m+1 \in 
    \Big\{\sum_{v\in C} s_v\, :\, \text{$C$ is a cycle in $G$ and $s_v \in S_v$}\Big\}. \]
Therefore $X$ is mixing by \cref{thm:mixing}, and $X$ has the specification property by \cref{thm:specification}.

Suppose $\max \Delta(S_1) = N$. There are uncountably many ways to choose two sets $A,B\subseteq \N$ such that $S_1 = A \sqcup B$ and both $\Delta(A)$ and $\Delta(B)$ are bounded by $3N$. For any such choice, vertex cloning at vertex $1$ with respect to the decomposition $S_1 = A \sqcup B$ produces an $\mathcal{S}$-graph shift on $2\lceil \lambda\rceil - 1$ vertices which is mixing and has entropy $\log \lambda$ by \cref{thm:cloning_preserves_entropy}. Further, since $\Delta(A)$ and $\Delta(B)$ are bounded by $3N$, the shift space has the weak specification property by \cref{thm:almost_spec} and the specification property by \cref{thm:specification}. Since there are uncountably many partitions $A \sqcup B$, \cref{thm:countable-conjugates} guarantees uncountably many pairwise non-conjugate $\mathcal{S}$-graph shifts on $2\lceil \lambda\rceil - 1$ vertices with entropy $\log \lambda$ and the specification property.

If $q > 2\lceil \lambda\rceil - 1$, then vertex clone as many times as necessary to acquire $q$ vertices, maintaining the specification property and entropy at each step.
\end{proof}

Since the entropy of an $\mathcal{S}$-graph shift on $q$ vertices is at most $\log q$, the number $2 \lceil \lambda \rceil -1$ in \cref{thm:every-entropy} can be decreased by at most a factor of 2. The minimal number of required vertices remains unclear, even for the weaker statement in \cref{thm:shift_of_every_entropy}.

\begin{question}
    What is the minimal number of vertices in an $\S$-graph shift with entropy $\log \lambda$?
\end{question}

\section{Intrinsic ergodicity}\label{sec:intrinsic-ergodicity}

There is a second, distinct notion of entropy in dynamical systems. Given a shift space $X$, let $\mathcal M(X)$ denote the set of all probability measures on $X$ that are invariant under the shift map $\sigma$. Though its formula is more complicated than for topological entropy, it's possible to assign a number $h_\mu(\sigma)$, called the \textit{measure-theoretic entropy}, to each measure $\mu \in \mathcal M(X)$. These two notions of entropy are related through the so-called \textit{variational principle}:
\[
    h(X) = \sup\,\{ h_\mu(\sigma) : \mu \in \mathcal M(X)\}.
\]
For a general dynamical system, this supremum may not be attained by any probability measure; for a shift space, however, it always is. When there is a unique measure that attains the supremum, the shift space is called \textit{intrinsically ergodic}.

Bowen showed in 1974 that every shift space with the specification property is intrinsically ergodic \cite{specification-ergodic}. Climenhaga and Thompson \cite{intrinsic-ergodicity} extended Bowen's result to prove that every shift space that is ``close enough'' to having the specification property is also intrinsically ergodic. In particular, they prove that every $S$-gap shift, as well as any factor of an $S$-gap shift, is intrinsically ergodic; Matson and Sattler \cite{S-limited} utilize their results to prove that every factor of an ordered $\S$-limited shift is, as well. In this section, we prove that this result extends much farther:

\begin{theorem}\label{thm:intrinsic-spec}
    Every factor of an $\S$-graph shift is intrinsically ergodic.
\end{theorem}

In particular, every factor of an unordered $\S$-limited shift is intrinsically ergodic, which extends the corresponding statement for ordered $\S$-limited shifts proved in \cite{S-limited}. To prove \cref{thm:intrinsic-spec}, we will state the relevant results from \cite{intrinsic-ergodicity}, which hinge on the following extension of the specification property to subsets of the space. Recall that the \textit{language} of a shift space is the set of all finite words that appear in any element of $X$, and $\mathcal B_n(X)$ is the set of all $n$-letter words in the language.

\begin{definition}
    Let $X$ be a shift space with language $\mathcal L$. A set $\Gamma \subseteq \mathcal L$ has the \textit{specification property} in the language $\mathcal L$ if there exists a natural number $t$ for which the following is true: For all natural numbers $m \in \N$ and words $w_1, w_2,\dots, w_m \in \Gamma$, there are words $v_1,v_2,\dots,v_{m-1} \in \mathcal L$ such that $w_1v_1w_2v_2\cdots w_{m-1}v_{m-1}w_m \in \mathcal L$ and $\lvert v_i\rvert = t$ for every $1 \leq i \leq m-1$. If this statement is true after relaxing the equality $\lvert v_i \rvert = t$ to the inequality $\lvert v_i\rvert \leq t$, then $\Gamma$ has the \textit{weak specification property}.
\end{definition}

\noindent
When $\Gamma = \mathcal L$, these definitions coincide with those presented in \cref{sec:dynamical-properties}.

Climenhaga and Thompson's results in \cite{intrinsic-ergodicity} establish intrinsic ergodicity for shifts whose languages can be decomposed into a prefix set, a ``good'' core set, and a suffix set. More precisely, they proved:

\begin{theorem}\label{thm:climenhaga}
    Suppose that $X$ is a shift space with language $\mathcal L$ for which there are three nonempty sets $C^p, C^s, \Gamma \subseteq \mathcal L$ so that
    \[
        \mathcal L = \{xyz \in \mathcal L : x \in C^p \text{, } y \in \Gamma\text{, and } z \in C^s\}.
    \]
    Further suppose that:
    \begin{enumerate}[label=\textup{(\Roman*)}]
        \item $\Gamma$ has the weak specification property.
        \item $h(C^p \cup C^s) < h(X)$, where $h(C) = \lim\limits_{n\to\infty} \frac{1}{n} \log \lvert C \cap \mathcal B_n(X)\rvert$.
        \item Setting $\Gamma(m) = \{uvw : u \in C^p\text{, } v \in \Gamma\text{, and } w \in C^s \text{ with } \lvert u \rvert, \lvert w\rvert \leq m \}$, the following is true: For every $m \in \N$, there is a natural number $r$ such that, for every $v \in \Gamma(m)$, there are words $u,w \in \mathcal L$ with at most $r$ letters for which $uvw \in \Gamma$.
    \end{enumerate}
    Then $X$ is intrinsically ergodic. Moreover, if $\tilde{X}$ is a factor of $X$, then there is a decomposition of its language into sets $\tilde{C}^p$, $\tilde{\Gamma}$, and $\tilde{C}^s$ such that Conditions (I) and (III) still hold and $h(\tilde{C}^p \cup \tilde{C}^s) \leq h(C^p \cup C^s)$.
\end{theorem}

With this result in hand, it is not too hard to prove that every $\S$-graph shift is intrinsically ergodic.

\begin{proof}[Proof of \cref{thm:intrinsic-spec}]
    We first find a decomposition for every $\S$-graph shift that satisfies Conditions (I), (II), and (III). Let $\mathcal G = (G,\mathcal S)$ be an essential set-equipped graph and consider the shift space $X = X(\mathcal G)$. We set
    \begin{align*}
        \Gamma &= \{ u_1^{s_1}\cdots u_n^{s_n} : (u_i,u_{i+1}) \in E(G)\text{ and } s_i \in S_{u_i} \text{ for every i}\}\\
        C^p = C^s &= \{ u^n : u \in V(G) \text{ and } n \in \N_0 \}.
    \end{align*}
    \begin{enumerate}[label=\textup{(\Roman*)}]
        \item For each $u \in V(G)$, set $m_u = \min S_u$; then define, for each pair $u,v \in V(G)$,
        \[
            \delta(u,v) = \min\Big\{ {\textstyle\sum_i m_{w_i}} : (u,w_1,\dots,w_n,v) \text{ is a walk in } G\Big\}.
        \]
        We claim that the set $\Gamma$ has the weak specification property with $t = \max_{u,v \in V(G)} \delta(u,v)$. Suppose $x,y \in \Gamma$, where $x$ ends with a block of the letter $u$ and $y$ begins with a block of the letter $v$. If $(u,w_1,\dots,w_n,v)$ is a walk in $G$, then $xw_1^{m_{w_1}}\cdots w_n^{m_{w_n}}y$ is another word in $\Gamma$, and we can always choose a walk for which $\sum_i m_{w_i} \leq t$.
        \item This is simple: $\lvert (C^p \cup C^s) \cap \mathcal B_n(X)\rvert = \lvert V(G)\rvert$ for every $n \geq 1$, so $h(C^p \cup C^s) = 0$.
        \item We can choose
        \[
            r = \max_{u \in V(G)} \min\!\big(S_u \setminus \{1,2,\dots,m-1\}\big).
        \]
        Since every word in $\Gamma(m)$ begins and ends with a partial block of at most $m$ letters, we need to add at most $r$ letters on either end to complete the block.
    \end{enumerate}
    By the second part of \cref{thm:climenhaga}, any factor of $X$ also satisfies conditions (I), (II), and (III). Therefore $X$ and each of its factors is intrinsically ergodic.
\end{proof}

Climenhaga and Thompson also note \cite[Proposition 2.4]{intrinsic-ergodicity} that every shift space that satisfies Condition (III) and an upgraded version of Condition (I) with the specification property in place of the weak one is either a single periodic orbit or has positive entropy. As we proved in \cref{thm:zero_entropy_shifts}, the same conclusion holds for every irreducible $\S$-graph shift.

\section{Zeta function}\label{sec:zeta}

Recall that the zeta function of a shift space $X$ is the power series defined as
\[
    \zeta_X(t) = \exp\!\left(\,\sum_{n=1}^\infty \frac{p_n(X)}{n}t^n\!\right).
\]
This section proves \cref{thm:big-zeta}, establishing an explicit formula for the zeta function of $\S$-graph shifts.

Before we start, it will be useful to introduce some new definitions. Recall that a walk in a set-equipped graph is composed of a walk $W = (v_1,\dots,v_n)$ on $G$ together with a selection $(s_1,\dots,s_n)$ of an element $s_i \in S_{v_i}$ at each vertex. In some cases, a closed walk $C$ may be formed by concatenating a single closed walk with itself several times. We let $\fc(C)$ denote the smallest closed walk that can form $C$ via repeated concatenation; we call this walk the \textit{fundamental closed walk} of $C$. (This fundamental closed walk may still repeat vertices; see \cref{ex:fund-cl-walk}.) Recall that the \textit{size} of a closed walk is the sum of the elements of the $S_i$ selected in the walk and is denoted $\lvert W\rvert$; the \textit{length} of $W$ is the number of vertices (including repetition) in the walk; this is denoted $\ell(W)$.

\begin{example}\label{ex:fund-cl-walk}
    In the closed walk $(a,d,a,d,a,d,a,d)$ with selections $(1,4,3,2,1,4,3,2)$, the fundamental closed walk is $(a,d,a,d)$ with selections $(1,4,3,2)$, so $\ell(\fc(C)) = 4$. The size of $\fc(C)$ is $\lvert \fc(C)\rvert = 1 + 4 + 3 + 2 = 10$.
\end{example}

We use $\mathcal C$ to denote the set of all closed walks (with at least two vertices) in $\mathcal G$. This first lemma shows that a sum of powers of the generating matrix is already very similar to the logarithm of the zeta function.

\begin{lemma}\label{thm:trace-powers-zeta-fn}
    If $\mathcal G$ is a set-equipped graph and $X = X(\mathcal G)$ is the $\S$-graph shift associated with it, then
    \begin{equation}\label{eq:periodic-gen-fn}
        \sum_{n=1}^\infty \frac{1}{n}\tr(B_\mathcal G(t)^n)
        = \sum_{n=1}^\infty\, \frac{p_n(X)-p_1(X)}{n} t^n.
    \end{equation}
\end{lemma}
\begin{proof}
    Let $P \subseteq \mathcal L(X)$ be the set of all words $w$ of at least two letters such that $w^\infty \in X$. (Recall that $w^\infty$ is the periodic point formed by concatenating infinitely many copies of $w$.) The points in $P$ can be divided into equivalence classes where $u \sim w$ if $\lvert u \rvert = \lvert w\rvert$ and $\sigma^n(u^\infty) = w^\infty$ for some $n \in \N_0$. The least period of $w$ is the number of elements in the equivalence class of $w$.
    
    Since the label of each vertex in $\mathcal G$ is unique, each equivalence class corresponds to exactly one closed walk in $\mathcal G$ and vice versa. Therefore
    \[
        \tr(B_\mathcal G(t)^n)
            = \sum_{\mathclap{\substack{C \in \mathcal C\\\ell(C) = n}}}\, \ell\big(\!\fc(C)\mspace{-2.5mu}\big)\, t^{|C|},
    \]
    since there are $\ell\big(\!\fc(C)\mspace{-2.5mu}\big)$ places to ``start'' each cycle. Since $\ell(C) / \ell\big(\!\fc(C)\mspace{-2.5mu}\big) = \lvert C\rvert / \lvert \fc(C)\rvert$ (both are the number of times $\fc(C)$ must be concatenated to form $C$), we have
    \[
        \tr(B_\mathcal G(t)^n)
            = \sum_{\mathclap{\substack{C \in \mathcal C\\\ell(C) = n}}}\, n\, \lvert\fc(C)\rvert\, \frac{t^{\lvert C\rvert}}{\lvert C\rvert}.
    \]
    Summing over $n$ and isolating the coefficient of $t^r$, we get
    \[
        [t^r]\sum_{n=1}^\infty \frac{1}{n}\tr(B_\mathcal G(t)^n)
            = \frac{1}{r}\, \sum_{\mathclap{\substack{C \in \mathcal C\\|C| = r}}}\, \lvert \fc(C)\rvert.
    \]
    Because $\lvert \fc(C)\rvert$ is the total number of periodic words that are associated to the cycle $\fc(C)$, the sum on the right side is the number of words in $X$ with period $r$ and least period at least 2. In other words,
    \[
        \sum_{n=1}^\infty \frac{1}{n}\tr(B_\mathcal G(t)^n)
        = \sum_{r=1}^\infty\, \frac{p_r(X)-p_1(X)}{r} t^r.\qedhere
    \]
\end{proof}

\begin{lemma}\label{thm:trace-sum-to-det}
    For any square matrix $A$, 
    \begin{equation}\label{eq:trace-sum}
        \exp\!\left(\,\sum_{n=1}^\infty \frac{1}{n} \tr(A^n)\!\right) = \frac{1}{\det(I - A)}.
    \end{equation}
\end{lemma}
\begin{proof}
    Suppose that $A$ is a $k \times k$ matrix with eigenvalues $\lambda_1, \dots, \lambda_k$. The left sum of \eqref{eq:trace-sum} simplifies to
    \[
        \sum_{n=1}^\infty \frac{\lambda_1^n + \cdots + \lambda_k^n}{n}
        = \sum_{i=1}^k \sum_{n=1}^\infty \frac{\lambda_i^n}{n}
        = \sum_{i=1}^k -\ln(1-\lambda_i).
    \]
    Then
    \[
        \exp\!\left(\,\sum_{n=1}^\infty \frac{1}{n} \tr(A^n)\!\right)
        = \exp\!\left(\,\sum_{i=1}^k -\ln(1-\lambda_i)\!\right)
        = \prod_{i=1}^k \frac{1}{1-\lambda_i},
    \]
    whose denominator is the characteristic polynomial of $A$ evaluated at 1, which is $\det(I - A)$.
\end{proof}

\begin{reptheorem}{thm:big-zeta}
    If $\mathcal G$ is a set-equipped graph and $X = X(\mathcal G)$ is the $\S$-graph shift associated with it, then
    \begin{equation}\label{eq:zeta-function-primitive-s-graph}
        \zeta_X(t) = \frac{1}{(1-t)^{p_1(X)}\det\!\big(I - B_{\mathcal G}(t)\big)}.
    \end{equation}
\end{reptheorem}
\begin{proof}
    We calculate using \cref{thm:trace-powers-zeta-fn}:
    \begin{align*}
        \log \zeta_X(t)
        &= \sum_{n=1}^\infty \frac{p_n(X)}{n} t^n\\
        &= \sum_{n=1}^\infty \frac{1}{n}\tr\big(B_\mathcal G(T)^n\big) + \sum_{n=1}^\infty \frac{p_1(X)}{n}t^n\\
        &= \sum_{n=1}^\infty \frac{1}{n}\tr\big(B_\mathcal G(T)^n\big) - p_1(X)\ln(1-t).
    \end{align*}
    Taking the exponent of both sides and applying \cref{thm:trace-sum-to-det} yields the formula.
\end{proof}

Since the number $p_1(X)$ is easy to calculate---it's simply the number of vertices $u \in V(G)$ for which $S_u$ is infinite---\cref{thm:big-zeta} shows that $\det\!\big(I - B_\mathcal G(t)\big)$ is effectively as powerful an invariant for $\S$-graph shifts as the zeta function.

We can use \cref{thm:big-zeta} to find a simple explicit formula for the zeta function of ordered $\S$-limited shifts:

\begin{proposition}\label{thm:ordered-s-limited-zeta}
    The zeta function of the ordered $\S$-limited shift $X = \vv{X}(S_1,\dots,S_n)$ is given by
    \[
        \zeta_X(t)^{-1} = (1-t)^{p_1(X)} \bigg( 1 - \prod_{i = 1}^n \sum_{s \in S_i} t^s \bigg).
    \]
\end{proposition}
\begin{proof}
    When expanding $\det\!\big(I - B_\mathcal G(t)\big)$, where $\mathcal G$ is the usual cycle presentation for an ordered $\S$-limited shift, there are only two nonzero terms: The product of the 1's on the diagonal and the product of the off-diagonal $-H_u(t)$ terms.
\end{proof}

With a small amount of calculation, \cref{thm:ordered-s-limited-zeta} itself specializes to the zeta functions of $S$-gap shifts, a formula that was previously published in \cite{computations-S-gap}. It also easily implies Theorem 5.1 from \cite{S-limited} (which itself generalizes Theorem 5.1 in \cite{S-prime}), stated in slightly different terms as the following corollary. We use the notation $\{\!\{\,\cdot\,\}\!\}$ to denote a multiset.

\begin{corollary}\label{thm:ordered-S-zeta-function}
    If $\vv{X}(S_1,\dots,S_n) \cong \vv{X}(T_1,\dots,T_m)$, then $\{\!\{\sum_i S_i\}\!\} = \{\!\{\sum_i T_i\}\!\}$.
\end{corollary}
\begin{proof}
    Since the shift spaces are conjugate, their zeta functions are equal. In particular, the power series $\prod_{i=1}^n \sum_{s \in S_i} t^s$ and $\prod_{i=1}^n \sum_{r \in T_i} t^r$ are equal. Equating their coefficients yields the corollary.
\end{proof}

Finally, \cref{thm:big-zeta} also contains the formula for the zeta functions of SFTs, as every SFT is conjugate to a vertex shift $X_G$ for some graph $G$. If $A$ is the adjacency matrix for $G$, then (see Proposition 6.4.6 in \cite{Lind-Marcus})
\[
    \zeta_{X_G}(t) = \frac{1}{\det(I - tA)}.
\]
On the other hand, we can represent $X_G$ as an $\S$-graph shift $X(G,\mathcal S)$ with $S_u = \{1\}$ for each $u \in V(G)$. Then $B_{\mathcal G}(t) = tA$, so \cref{thm:big-zeta} reduces to the classic formula for SFTs.

\section{Appendix: Sequestered proofs}\label{sec:appendix}

\titleformat{\subsubsection}{}{\thesubsubsection.}{4pt}{\bfseries #1\vspace{-1ex}}{}

\subsection{Dynamical properties}

\subsubsection*{\cref{thm:SFT}}
\begin{proof}
    Let $X = X(G,\mathcal{S})$. First suppose each $S_i$ is either finite or cofinite and define
    \[ F_0 = \{ ij : \text{$i$ and $j$ are not adjacent in $G$} \}. \]
    For each $u \in V(G)$, if $S_u$ is finite with maximum element $m$, we define
    \[ F_u = \{ au^kb\, :\, a,b\in V(G)\setminus\{u\} \text{ and } k \in \{1,2,\dots,m\}\setminus S_u \} \cup \{ u^{m+1} \}; \]
    otherwise $S_u$ is cofinite and we define
    \[ F_u = \{ au^kb\, :\, a,b \in V(G)\setminus\{u\} \text{ and } k \in \N \setminus S_u \}. \]
    Each set $F_u$ is finite. Setting $\mathcal{F} = F_0 \cup \Big(\bigcup_{u \in V(G)} F_u\Big)$, we have $X = X_\mathcal{F}$, which shows that $X$ is an SFT.
    
    Now suppose that $S_u$ is neither finite nor cofinite for some $u \in V(G)$. No finite list of forbidden words will disallow all words in $\{ au^kb\, :\, a,b \in V(G)\setminus\{u\} \text{ and } k\not\in S_u\}$ and simultaneously allow all words in $\{ au^kb\, :\, (a,u), (u,b) \in E(G) \text{ and } k\in S_u\}$. Therefore $X$ is not an SFT.
\end{proof}

\subsubsection*{\cref{thm:sofic}}
\begin{proof}
    ($\Rightarrow$) Let $X = X(G,\mathcal S)$ and $(a,b) \in E(G)$. If $S_b$ is finite, then $\Delta(S_b)$ is eventually constant and therefore eventually has period 1. Otherwise, since $X$ is sofic, it has only finitely many follower sets by \cref{thm:sofic-follower-sets}. Therefore not all follower sets of the form $F(ab^k)$ are distinct, so there is a pair of integers $m < n$ so that $F(ab^m) = F(ab^n)$. This implies that $\Delta(S_b)$ eventually has period $n-m$.
    
    ($\Leftarrow$) Let $a,b \in V(G)$. If $\alpha ab^n \in \mathcal L(X)$ for some word $\alpha$, then $F(\alpha ab^n) = F(ab^n)$; so, to show that $X$ has only finitely many distinct follower sets, we need only consider follower sets of the form $F(ab^n)$ or $F(b^n)$. Assume that $\Delta(S_b)$ eventually has period $m_b$, that is, there exists an $N_b \in \mathbb{N}$ so that $s \in S_i$ if and only if $s-m_b \in S_i$ for every $s \geq N_b$. Then $F(ab^s) = F(ab^{s-m_b})$ for every $s \geq N_b$, so every follower set in $X$ has one of the following two forms:
    \begin{enumerate}
        \item $F(b^k)$ with $k \in [0,N_b] \cap S_b$ or
        \item $F(ab^k)$ with $k \in [0,N_b] \cap S_b$ and $(a,b) \in E(G)$.
    \end{enumerate}
    Since there are only finitely many such follower sets, $X$ is sofic.
\end{proof}

\subsubsection*{\cref{thm:mixing}}
\begin{proof}
    ($\Rightarrow$) Let $X = X(G,\mathcal{S})$ and $T = \left\{\sum_{v\in C} s_v\, :\, \text{$C$ is a cycle in $G$ and $s_v \in S_v$}\right\}$. First, suppose that $X$ is mixing. For any two letters $a,b \in V(G)$, there is a word $\alpha \in \mathcal L(X)$ for which $a\alpha b \in \mathcal L(X)$, so there is a path from $a$ to $b$. Therefore $G$ is irreducible. Let $(a,b)$ be an edge in $G$. Because $X$ is mixing, there is an $n \in \mathbb{N}$ with two words $\omega_1 \in \mathcal{B}_n(X)$ and $\omega_2\in\mathcal{B}_{n+1}(X)$ such that $b\omega_1 a,\ b\omega_2 a \in \mathcal L(X)$. We may choose $\omega_1$ and $\omega_2$ so that both have the form $b^{k_0-1}c_1^{k_1}\cdots c_m^{k_m}a^{k_{m+1}-1}$ for some $k_i \in S_{c_i}$, where $c_0=b$, $c_{m+1} = a$, and $(c_i,c_{i+1})\in E(G)$. Then $b\omega_1 a \in \mathcal{B}_{n+2}(X)$ and $b\omega_2 a\in\mathcal{B}_{n+3}(X)$.

    Both $b\omega_1 a$ and $a\omega_2 b$ correspond to closed walks in $G$ since $(a,b)$ is an edge. Any closed walk can be decomposed into cycles. Since $b\omega_1 a \in \mathcal{B}_{n+2}(X)$ and $b\omega_2 a\in\mathcal{B}_{n+3}(X)$, both $n+2$ and $n+3$ can be expressed as sums of elements of $T$. It follows that the greatest common divisor of the elements of $T$ is $1$.

    ($\Leftarrow$) Now suppose that $\gcd(T)=1$ and $G$ is irreducible. There exists some $N$ such that for all $n \geq N$, we can write $n$ as a sum of elements of $T$. Let $(a,b),\ (c,a) \in E(G)$ and $\omega$ be a word comprised of full blocks that is obtained from a walk $w$ in $G$ that begins at $b$, ends at $c$, and visits every vertex at least once. We first show that for all $n\geq N + \lvert \omega\rvert$, there is a word in $\mathcal{B}_n(X)$ consisting of concatenated full blocks. We will use this to prove that $X$ is mixing.

    If $n \geq N + \lvert\omega\rvert$, then $n-\lvert\omega\rvert$ can be written as a sum of elements in $T$, which corresponds to a multiset of cycles in $G$. By inserting these cycles into the walk $w$, we obtain a walk $w_n$ and its corresponding word $\omega_n$ of length $n$ that begins with a full block of $b$ and ends with a full block of $c$. For every word $\alpha \in \mathcal L(X)$, let $\xi^\smallrightarrow_\alpha$ denote a word of shortest length that begins with $\alpha$ and ends with a full block of $a$; let $\xi^\smallleftarrow_\alpha$ denote a word of shortest length that begins with a full block of $a$ and ends with $\alpha$.

    Given any words $\alpha,\beta \in \mathcal L(X)$ and
    $n \geq N + \lvert \omega\rvert + \lvert \xi^{\smallrightarrow}_\alpha \rvert + \lvert \xi^{\smallleftarrow}_\beta \rvert$, the word $\xi_\alpha^{\smallrightarrow}\, \omega_{n- \lvert \xi^{\smallrightarrow}_\alpha \rvert - \lvert \xi^{\smallleftarrow}_\beta\rvert}\, \xi^\smallleftarrow_\beta \in \mathcal L(X)$ begins with $\alpha$ and ends with $\beta$. So the shift space $X$ is mixing.
\end{proof}

\subsubsection*{\cref{thm:almost_spec}}
\begin{proof}
($\Rightarrow$) Let $X = X(G,\mathcal{S})$. If $G$ is not irreducible, so that there there exist $a,b \in V(G)$ with no path from $a$ to $b$, then there is no word $\gamma \in \mathcal{L}(X)$ such that $a\gamma b \in \mathcal{L}(X)$. So $X$ does not have the weak specification property.

Now suppose that $G$ is irreducible but $\sup \Delta(S_a) = \infty$ for some $a\in V(G)$. Let $(a,b)$ be an edge in $G$ and $N > 0$. Listing the elements of $S_a$ as $s_1 < s_2 < \cdots$, we may choose $k$ so that $n_{k+1} - n_k > N+1$. For any $b \in V(G) \setminus \{a\}$, there exists no word $\gamma\in\mathcal{B}_N(X)$ such that $a^{n_k + 1} \gamma b \in \mathcal L(X)$, so $X$ does not have the weak specification property.

($\Leftarrow$) Assume $G$ is irreducible and $\sup \Delta(S_u)$ is finite for each $u \in V(G)$, and let $d = \max_u\!\big(\!\sup \Delta(S_u)\big)$. For each $a,b\in V(G)$, choose some word $\tau_{a,b} \in \mathcal L(X)$ that does not begin with $a$ and does not end with $b$ such that $a\,\tau_{a,b}\, b \in \mathcal{L}(X)$. We claim that $N = \max \{ \lvert \tau_{a,b}\rvert : a,b \in V(G)\} + 2d$ satisfies the criterion for the weak specification property as follows.

Let $\alpha, \beta \in \mathcal{L}(X)$, and suppose $\alpha$ ends with the block $a^{k_1}$ and $\beta$ begins with the block $b^{k_2}$. There exist positive integers $s_1 \in S_a$ and $s_2 \in S_b$ such that $0 \leq s_1 - k_1 \leq d$ and $0 \leq s_2 - k_2 \leq d$. Then $\alpha a^{s_1-k_1}\,\tau_{a,b}\,b^{s_2-k_2}\beta \in \mathcal L(X)$. Since $\alpha$ and $\beta$ were arbitrary, $X$ has the weak specification property.
\end{proof}

\subsubsection*{\cref{thm:specification}}
\begin{proof}
($\Rightarrow$) Let $X = X(G,\mathcal{S})$. First suppose that $X$ has the specification property. Certainly $X$ has the weak specification property. Let $N$ be the integer guaranteed by the specification property. Given any two words $\alpha, \beta \in \mathcal L(X)$ and an integer $n \geq N$, choose some $\tau \in \mathcal B_{n-N}(X)$ with $\alpha\tau \in \mathcal L(X)$. There is a word $\gamma \in \mathcal B_N(X)$ so that $\alpha\tau\gamma\beta \in \mathcal L(X)$. Since $\alpha$ and $\beta$ were arbitrary, $X$ is mixing.

($\Leftarrow$) Now suppose that $X$ is mixing and has the weak specification property. By \cref{thm:almost_spec}, the integer $d = \max_i\!\big(\!\sup \Delta (S_i)\big)$ is finite. For each $a,b\in V(G)$, choose an integer $N(a,b)$ such that for all $n\geq N(a,b)$, there exists some $\tau \in \mathcal{B}_n(X)$ that does not begin with $a$ nor end with $b$ and $a\tau b\in \mathcal L(X)$. The integers $N(a,b)$ exist because $X$ is mixing. Then $N = \max\{N(a,b) : a,b\in V(G)\} + 2d$ satisfies the criteria for the specification property as follows.

Let $\alpha,\beta \in \mathcal L(X)$. Suppose that $\alpha$ ends with the block $a^{k_1}$ and $\beta$ begins with the block $b^{k_2}$. There exist positive integers $s_1 \in S_a$ and $s_2 \in S_b$ such that $0 \leq s_1 - k_1 \leq d$ and $0 \leq s_2 - k_2 \leq d$. We can choose a word $\tau \in \mathcal L(X)$ with $\lvert\tau\rvert = N - (s_1 - k_1) - (s_2 - k_2)$ such that $a^{s_1-k_1}\tau b^{s_2-k_2} \in \mathcal L(X)$, which implies that $\alpha a^{s_1-k_1}\tau b^{s_2-k_2}\beta \in \mathcal L(X)$.
\end{proof}

\subsection{Eigenvalue lemmas}

\subsubsection*{\cref{thm:perron-frobenius}}
\begin{proof}
Suppose that $A$ is an $m\times m$ matrix. By a standard result from linear algebra, we may relabel the indices of $A$ so that $A$ is block upper-triangular and each block in the main diagonal is irreducible. Let $C_1, \dots, C_k$ denote the blocks on the main diagonal of $A$ and $\rho(C_i)$ be the spectral radius of $C_i$. Let $t$ be an index for which $\rho(C_t)$ is maximal; since the spectrum of $A$ is the union of the spectra of the $C_i$, we have $\rho = \rho(C_t)$. For each $1 \leq i \leq k$, let $c_i$ and $d_i$ be the constants supplied by \cref{thm:perron-frobenius-original} when applied to $C_i$, and set $\bar{d} = \max_i d_i$; also set $\alpha = \max\{1, \max_{i,j} A_{i,j}\}$.

\cref{thm:perron-frobenius-original} implies that
\[  \sum_{i,j} (A^n)_{i,j}
    \geq \sum_{i,j} (C_t^n)_{i,j}
    \geq c_t\rho(C_t)^n
    = c_t\rho^n, \]
which proves the first inequality. Let $K$ be the complete graph with loops on the vertex set $\{1,2,\dots, m\}$ (that is, $E(K) = V(K) \times V(K)$). Given a walk $w = (v_1\dots,v_s)$ on $K$, we define
\[ f(w) = A_{v_1,v_2}A_{v_2,v_3}\cdots A_{v_{s-1},v_s} \]
and note that $\sum_{i,j} (A^n)_{i,j} = \sum_{\lvert w \rvert = n+1} f(w)$. By a minor abuse of notation, we let $V(C_i)$ denote the set of indices that determine the matrix $C_i$ in $A$ and call these the \textit{blocks} of $V(K)$.

If $f(w) \neq 0$, then there are at most $k-1$ indices $r$ so that $v_r$ and $v_{r+1}$ are in different blocks. Therefore, we decompose $w$ as $w_1e_1w_2e_2\cdots e_{j-1}w_j$, where each $e_i$ is an edge between blocks and each $w_i$ is a walk inside a single block. An \textit{arrangement} of $j-1$ inter-block edges is a $j$-tuple $(\gamma_1,\dots,\gamma_j)$ of positive integers that sum to $n+1$ and a $(j-1)$-tuple $(e_1,\dots,e_{j-1})$ of inter-block edges of $K$. If $R$ is an arrangement of inter-block edges, we say that a walk $w$ in $K$ \textit{follows} the arrangement $R$ if $w$ can be decomposed as $w = w_1e_1\dots e_{j-1}w_j$ where each $w_i$ is a walk in a single block of $K$ with exactly $\gamma_i$ vertices.

An arrangement is entirely determined by the inter-block edges and their positions in the walk. Since there are $m^2$ edges in $K$ (not all of which are necessarily inter-block) and at most $n$ places for an inter-block edge to occur in a walk of $n+1$ vertices, the number of arrangements for the inter-block edges is at most
\[
    m^2n + (m^2n)^2 + \cdots + (m^2n)^{k-1} \leq (m^2n)^k.
\]
Fix one such arrangement $R$ with edge-tuple $(e_1,\dots, e_{j-1})$ and integer-tuple $(\gamma_1,\dots,\gamma_j)$. Let $C_{r_i}$ be the block that contains the tail of edge $e_i$ and $C_{r_j}$ be the block that contains the head of $e_j$, and let $W_R$ denote the set of all walks of length $n+1$ that follow the arrangement $R$. By \cref{thm:perron-frobenius-original}, we have
\begin{align*}
    \sum_{w \in W_R} f(w)
    &\leq \alpha^k \sum_{w \in W_R} f(w_1)\cdots f(w_j)\\
    &\leq \alpha^k \prod_{i=1}^j \sum_{s,t} (A_{C_{r_i}}^{\gamma_i-1})_{s,t}\\
    &\leq \alpha^k \prod_{i=1}^j \bar{d} \rho^{\gamma_i-1}\\
    &\leq (\alpha\bar{d})^k \rho^n.
\end{align*}
Therefore
\[ \sum_{i,j} (A^n)_{i,j} = \sum_{\lvert w \rvert = n+1} f(w)
    = \sum_{R} \sum_{w \in W_R} f(w)
    \leq (m^2n)^k (\alpha\bar{d})^k \rho^n. \]
So taking $d = (\bar{d}\alpha m^2n)^k$ finishes the proof.
\end{proof}

\subsubsection*{\cref{thm:lambda_is_increasing}}

{\noindent
\textit{Preliminaries.} %
}%
An \emph{irreducible component} of a graph $G$ is a maximal irreducible subgraph of $G$. If $G$ has irreducible components $I_1, \dots, I_k$, then the submatrices $[B_{\mathcal G}(x)]_{i,j \in I_r}$, for each $1 \leq r \leq k$, are called the \emph{irreducible components} of $B_{\mathcal G}(x)$. It is known (see, for example, \cite[Lemma 4.4.3]{Lind-Marcus}) that the spectral radius of $B_{\mathcal G}(x)$ (for $x > 0$) is equal to the maximum of the spectral radii of its irreducible components. Moreover, Theorem 4.4.7 of \cite{Lind-Marcus} implies that, for any two nonnegative $q \times q$ matrices $A$ and $B$ such that $A_{i,j} \leq B_{i,j}$ for each $1 \leq i,j \leq q$, with strict equality holding for at least one pair $(i,j)$ in each irreducible component, the spectral radius of $A$ is strictly less than the spectral radius of $B$.

\begin{proof}
    The spectral radius of a matrix is a continuous function of its entries.\footnote{This is a standard result that follows from a chain of continuous maps: the map from the matrix entries to the coefficients of the characteristic polynomial, the map from a complex polynomial to its set of roots (see, for example, \cite{continuous-roots}), and the map from the set of roots to its maximum complex norm.}  Since the entries of $B_{\mathcal G}(x)$ vary continuously with $x$, the function $\rho(x)$ is continuous.
    
    If $0 \leq x < y$, then $B_{\mathcal G}(x)_{i,j} \leq B_{\mathcal G}(y)_{i,j}$ for every $1\leq i,j\leq q$, with strict inequality holding for every pair $(i,j)$ such that $B_{\mathcal G}(y)_{i,j} \neq 0$. The preliminaries preceding this proof imply that $\rho(y) > \rho(x)$.
    
    Finally, we prove that the range of $\rho(x)$ is $[0,+\infty)$. Since $\rho(0) = 0$ and $\rho$ is continuous, it suffices to show that $\rho(x) \to \infty$ as $x \to \infty$. If $x \neq 0$, then $B_{\mathcal G}(x)$ is not the zero matrix, so $\rho(x) > 0$. For every $x \geq 0$, the inequality $B_{\mathcal G}(2x) \geq 2B_{\mathcal G}(x)$ holds entrywise. Again, the preliminaries imply that $\rho(2x) \geq 2\rho(x)$; therefore $\lim_{x \to \infty} \rho(x) = \infty$.
\end{proof}

\subsection{\texorpdfstring{$\beta$}{Beta}-expansion}

\subsubsection*{\cref{thm:mu_exp_of_1}}

\begin{proof}
Let $n = \lceil \beta\rceil-1$. We want to find a sequence $(x_i)_{i=1}^\infty$ such that
\begin{equation}\label{eq:beta-expansion}
    1 = \sum_{i=1}^k \frac{x_i}{\beta^i} + \frac{n}{\beta^{k+1}} + \frac{x_{k+1}}{\beta^{k+2}} + \frac{n}{\beta^{k+3}} + \frac{x_{k+2}}{\beta^{k+4}} + \cdots
\end{equation}
for some $k\in \N$. The sum
\[ \frac{n}{\beta^{k+1}} + \frac{n}{\beta^{k+3}} + \frac{n}{\beta^{k+5}} + \cdots = \frac{n/\beta^{k-1}}{\beta^2-1} \]
is less than 1 if $k$ is large enough. Fix such a $k$. We compute values of $(x_i)$ according to a greedy algorithm: After choosing values for $x_1,\dots,x_{m-1}$, we choose $x_m$ to be the largest value in $\{0,1,\dots,n\}$ such that, if $x_i = 0$ for every $i > m$, the right-hand side of \eqref{eq:beta-expansion} is at most 1. The resulting $\beta$-expansion represents a real number $x$, which is at most 1 by construction. Since we chose $x_i$ to be maximal at each step, we know that
\[ x + \frac{1}{\beta^{k+2i}} > 1 \]
whenever $x_{k+i} < n$. If there are infinitely many values of $i$ for which that inequality holds, then taking the limit as $i \to \infty$ shows that $x \geq 1$, in which case $(x_i)$ is a $\beta$-expansion of 1.

Now suppose there are finitely many $i$ such that $x_{i} < n$, and let $x_m$ be the last such term. The value of the tail after this point is 
\[  \sum_{i=m+1}^\infty \frac{n}{\beta^i}
    = \frac{n}{\beta^m(\beta-1)}
    \geq \frac{1}{\beta^m}, \]
contradicting the fact that $x_m$ is maximal. This means that $x_i = n$ for every $i \in \N$, so
\[  1 \geq \sum_{i=1}^\infty \frac{n}{\beta^i}
    = \frac{n}{\beta-1}\]
by construction. On the other hand, $\beta - 1 \leq n$, so $n/(\beta-1) \geq 1$. Therefore $(x_i)$ is a $\beta$-expansion of 1.

Since we chose $(x_i)$ by the greedy algorithm, we know that $x_1 > 0$. If $(x_i)$ does not contain two consecutive nonzero integers, then define
\[ y_i = \begin{cases}
x_1 - 1 &\text{if } i=1\\
x_i & \text{if } i > 1 \text{ and } x_i > 0\\
x_{i-1} & \text{if } i > 1 \text{ and } x_i = 0.
\end{cases} \]
Otherwise, set $y_i = x_i$ for all $i\in \N$. In either case, the sequence $(y_i)_{i=1}^\infty$ is a $\beta$-expansion of $1$ that satisfies the conditions.
\end{proof}

\vspace{3ex}

\noindent\phantomsection\addcontentsline{toc}{section}{Acknowledgements}%
{\large \textbf{Acknowledgments}}\\[0.3em]
I thank Elizabeth Sattler for suggesting an investigation into the entropy of unordered $\mathcal{S}$-limited shifts and her unfailingly insightful advice; Vaughn Climenhaga for several helpful comments; and the anonymous referees for their thorough review and trenchant suggestions.

\phantomsection
\addcontentsline{toc}{section}{References}
\bibliography{bibliography}
\bibliographystyle{custom-bibliography}

\vspace{2\baselineskip}

\noindent
\textsc{Travis Dillon}, \textit{email:} \texttt{dillont@mit.edu}

\end{document}